\documentclass[11pt]{amsart}
\usepackage{latexsym,amssymb,amsfonts,amsmath,graphicx,fullpage,url,color,tikz}
\usepackage[utf8]{inputenc}
\usepackage[vcentermath,enableskew]{youngtab}
\usepackage[all]{xy}
\usepackage{verbatim}
\usepackage{ytableau}

\newtheorem{theorem}{Theorem}[section]
\newtheorem{proposition}[theorem]{Proposition}
\newtheorem{lemma}[theorem]{Lemma}
\newtheorem{corollary}[theorem]{Corollary}
\theoremstyle{definition}
\newtheorem{definition}[theorem]{Definition}

\newtheorem{example}[theorem]{Example}

\theoremstyle{remark}
\newtheorem{remark}[theorem]{Remark}
\numberwithin{equation}{section}
\usepackage{multicol}\newcommand{\tabbedblock}[1]{
\begin{tabbing}
\hspace{.28in} \= \hspace{3.05in} \= \hspace{2.4in} \= \hspace{3.2in} \kill #1
\end{tabbing}
}

\newcommand{\thickhat}[1]{\mathbf{\hat{\text{$#1$}}}}
\newcommand*\circled[1]{\tikz[baseline=(char.base)]{
            \node[shape=circle,draw,inner sep=2pt] (char) {#1};}}

\author{Sara Solhjem and Jessica Striker}
\email{jessica.striker@ndsu.edu, sara.solhjem@ndsu.edu}

\address{North Dakota State University}

\title{Sign matrix polytopes from Young tableaux}

\begin{document}

\begin{abstract}
Motivated by the study of polytopes formed as the convex hull of permutation matrices and alternating sign matrices, we define several new families of polytopes as convex hulls of sign matrices, which are certain $\{0,1,-1\}$--matrices in bijection with semistandard Young tableaux. 
We investigate various properties of these polytopes, including their inequality descriptions, vertices, facets, and face lattices, as well as connections to alternating sign matrix polytopes and transportation polytopes. 
\end{abstract}

\maketitle
\tableofcontents

\section{Introduction}
\emph{Sign matrices} are defined as $\{0,1,-1\}$--matrices whose column partial sums are zero or one and whose row partial sums are nonnegative. Sign matrices were introduced by Aval~\cite{aval}, who showed they are in bijection with \emph{semistandard Young tableaux}. Young tableaux are well-loved objects for their nice combinatorial properties, including beautiful enumerative formulas, and nontrivial connections to Lie algebras, representation theory, and statistical physics \cite{BumpSchilling, fulton, KSS}. Aval used sign matrices to give a simple method for computing the \emph{left key} of a tableau by successively removing the negative ones from its corresponding sign matrix~\cite{aval}.

\emph{Alternating sign matrices} are $n \times n$ sign matrices with the additional properties that the rows and columns each sum to one and the row partial sums may not exceed one~\cite{MRRASM}. Alternating sign matrices were introduced by Robbins and Rumsey in their study of the $\lambda$-determinant~\cite{RobbinsRumsey}, with an enumeration formula conjectured by Mills, Robbins, and Rumsey~\cite{MRRASM}. The proof of this conjecture~\cite{zeilberger,kuperbergASMpf} was a major accomplishment in enumerative combinatorics in the 1990's. Alternating sign matrices are still a source of interest, in particular, with regard to intriguing open bijective questions involving plane partitions and connections to both the six-vertex model and various loop models in statistical physics~\cite{ZINNDPP,Bettinelli,Biane_Cheballah_1,BRESSOUDBOOK,razstrogpf2,razstrogpf,ProppManyFaces,razstrog,Striker_DPP,StrikerFPSAC2013,razstrogrow,STRIKERPOSET}. 

In \cite{striker}, the second author examined alternating sign matrices from a geometric perspective by defining and studying the polytope formed by taking the convex hull of all $n\times n$ alternating sign matrices, as vectors in $\mathbb{R}^{n^2}$. 
She studied various aspects of the alternating sign matrix polytope, including its dimension, facet count, vertices, face lattice, and inequality description. Independently, 
Behrend and Knight~\cite{behrend}  defined and studied the alternating sign matrix polytope. They proved the equivalence of the inequality and vertex descriptions,
computed the Ehrhart polynomials to $n=5$,
and studied lattice points in the $r$th dilate of the alternating sign matrix polytope, which they called {higher spin alternating sign matrices}.

In this paper, we extend this work by studying polytopes formed as convex hulls of sign matrices. We  define two new families of polytopes: $P(m,n)$ as the convex hull of all $m\times n$ sign matrices and $P(\lambda,n)$ as the convex hull of sign matrices in bijection with semistandard Young tableaux of a given shape $\lambda$ and entries at most $n$. 
If we, furthermore, fix the entries in the first column of the tableaux to be determined by the vector $v$, we obtain a polytope $P(v,\lambda,n)$, whose nonnegative part we show in Theorem~\ref{thm:transportation} is a transportation polytope.

\textbf{Our main results} include Theorems~\ref{thm:ineqthmshape}, \ref{thm:ineqthm}, and \ref{thm:v_ineqthmshape}, in which we have found the set of inequalities that determine $P(\lambda,n)$, $P(m,n)$, and $P(v,\lambda,n)$ by an extension of the proof technique von Neumann used to show that the convex hull of $n\times n$ permutation matrices, the $n$th Birkhoff polytope, consists of all $n\times n$ doubly stochastic matrices~\cite{vonneumann}. Other main results include vertex charactorizations (Theorems~\ref{thm:lambdavertex}, \ref{thm:mnvertex}, and \ref{thm:v_lambdavertex}), descriptions of the face lattices of these polytopes (Theorems~\ref{th:g_bijection}, \ref{thm:poset_iso}, \ref{th:g_bijection_shape}, and \ref{th:v_g_bijection_shape}), and enumerations of the facets (Theorems~\ref{thm:facets_mn} and \ref{thm:facets_lambda}). 

\textbf{Our outline is as follows.} In Section~\ref{sec:ssyt}, we refine Aval's bijection between semistandard Young tableaux and sign matrices to account for the tableau shape. In Section~\ref{sec:Plambda_n}, we define the polytope $P(\lambda,n)$ as the convex hull of all $\lambda_1\times n$ sign matrices corresponding to semistandard Young tableaux of shape $\lambda$ and entries at most $n$, prove its dimension, and show that the vertices are all the sign matrices used in the construction. In Section~\ref{sec:Pmn}, we define the polytope $P(m,n)$ as the convex hull of all $m\times n$ sign matrices, find its dimension and vertices. Then in Section~\ref{sec:ineq}, we prove Theorems~\ref{thm:ineqthmshape} and \ref{thm:ineqthm}, giving an inequality description of $P(\lambda,n)$ and $P(m,n)$ respectively. In Section~\ref{sec:Pmninq}, we prove facet counts for both polytope families (Theorems~\ref{thm:facets_mn} and \ref{thm:facets_lambda}). In Section~\ref{sec:facelattice} we give a description of the face lattices of these polytopes. In Section~\ref{sec:connections}, we describe how $P(m,n)$ and $P(\lambda,n)$ relate to each other and give connections to alternating sign matrix polytopes. In Section~\ref{sec:transportation}, we define another polytope $P(v,\lambda,n)$ as the convex hull of sign matrices in bijection with semistandard Young tableaux of shape $\lambda$, entries at most $n$, and first column given by $v$. We then prove Theorem~\ref{thm:transportation}, relating these polytopes to transportation polytopes.

\section{Semistandard Young tableaux and sign matrices}
\label{sec:ssyt}

In this section, we first define semistandard Young tableaux and sign matrices. We then discuss a bijection between them, due to Aval. We refine this bijection in Theorem~\ref{thm:MtoSSYT} to a bijection between semistandard Young tableaux with a given shape and sign matrices with prescribed row sums. 

We use the following notation throughout the paper.
\begin{definition}
\label{def:YD}
A \emph{partition} is a weakly decreasing sequence of positive integers $\lambda=[\lambda_1,\lambda_2,\ldots,\lambda_k]$. The positive integers $\lambda_i$ are called the \emph{parts} of the partition and $k$ is the \emph{length} of the partition.
A \emph{Young diagram} is a visual representation of a partition $\lambda$ as a collection of boxes, or cells, arranged in left-justified rows, with $\lambda_i$ boxes in row $i$. We will refer to a partition and its Young diagram interchangeably.

Let $\lambda'$ denote the \emph{conjugate partition} of $\lambda$, that is, the Young diagram defined by reflecting $\lambda$ about the diagonal. Note $k=\lambda'_1$.

The \emph{frequency representation} of $\lambda$ is the sequence $[a_1,a_2,\ldots,a_{\lambda_1}]$ where $a_i$ equals the number of parts of $\lambda$ equal to $i$. 
We may also denote $\lambda$ using \emph{exponential notation} as $[\lambda_1^{a_{\lambda_1}},\ldots,i^{a_i},\ldots,\lambda_k^{a_{\lambda_k}}]$.
\end{definition}

\begin{example}
The partition $\lambda=[6,3,3,1]$ has $k=4$ parts. The exponential notation for $\lambda$ is $[6,3^2,1]$ and its frequency representation is $[1,0,2,0,0,1]$. The conjugate partition is $\lambda'=[4,3,3,1,1,1]$. 
See Figure~\ref{fig:ydssyt}.
\end{example}

\begin{definition}
A \emph{semistandard Young tableau (SSYT)} is a filling of a Young diagram with positive integers such that the rows are weakly increasing and the columns are strictly increasing. See Figure~\ref{fig:ydssyt}.
\label{def:SSYT}
\end{definition}

\begin{figure}[htbp]
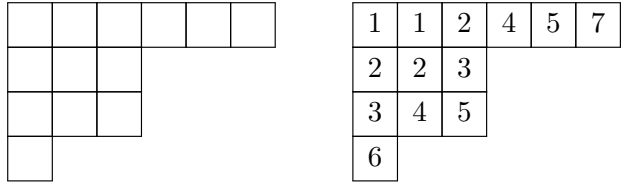

\begin{center}
$\begin{ytableau}
\text{  }&&&&&\\
&&\\
&&\\
\text{  }
\end{ytableau}$ \hspace{.3in} 
$\begin{ytableau}
1&1&2&4&5&7\\
2&2&3\\
3&4&5\\
6
\end{ytableau}$
\end{center}
\caption{A Young diagram of partition shape $\lambda=[6,3,3,1]$ and a semistandard Young tableau of the same shape.}
\label{fig:ydssyt}
\end{figure}

\begin{definition}
Let $SSYT(m,n)$ denote the set of {semistandard Young tableaux with at most $m$ columns 
and entries at most $n$}. 
\end{definition}

Gordon enumerated $SSYT(m,n)$ as follows.

\begin{theorem}[\cite{gordon}]  
\label{thm:Gordon}
The number of SSYT with at most $m$ columns and entries at most $n$ is 
\[\displaystyle\prod_{1\le i\le j\le n} \frac{m+i+j-1}{i+j-1}.\] 
\end{theorem}

\begin{definition}
Let $SSYT(\lambda,n)$ denote the set of {semistandard Young tableaux of partition shape $\lambda$ 
and entries at most $n$}. 
\end{definition}

For example, the tableau of Figure~\ref{fig:ydssyt} is in both $SSYT(6,n)$ and $SSYT([6,3,3,1],n)$ for any $n \geq 7$.

\smallskip
$SSYT(\lambda,n)$ is enumerated by Stanley's hook-content formula.

\begin{theorem} [\cite{rpstanley}]
The number of SSYT of shape $\lambda$ with entries at most $n$ is
\[  \displaystyle \prod_{u \in \lambda} \frac{n+c(u)}{h(u)} \]
where $c(u)$ is the content of the box $u$, given by $c(u)=i-j$ for  $u=(i,j)$, and $h(u)$ is the hook length of $u$, given by the number of squares directly below or to the right of $u$ (counting $u$ itself). 
\end{theorem}

Aval \cite{aval} defined a new set of objects, called sign matrices, which will be the building blocks of the polytopes that will be our main objects of study.

\begin{definition}[\cite{aval}]
\label{def:sm}
A \emph{sign matrix} is a matrix $M=\left(M_{ij}\right)$ with entries in $\left\{-1,0,1\right\}$ such that:
\begin{align}
\label{eq:sm1}
\displaystyle\sum_{i'=1}^{i} M_{i'j} &\in \left\{0,1\right\}, & \mbox{ for all }i,j. \\
\label{eq:sm2}
\displaystyle\sum_{j'=1}^{j} M_{ij'} &\geq 0, & \mbox{ for all }i,j.
\end{align}
\end{definition}

In words, the column partial sums from the top of a sign matrix equal either 0 or 1 and the partial sums of the rows from the left are non-negative. 

Aval showed that $m\times n$ sign matrices are in bijection with SSYT with at most $m$ columns and largest entry at most $n$ \cite[Proposition 1]{aval}.  
We now define the set of sign matrices we will show in Theorem~\ref{thm:MtoSSYT} to be in bijection with $SSYT(\lambda, n)$; this is a refinement of Aval's bijection. See Figure~\ref{fig:(3,3,1,1,1)} for an example of this bijection. 

\begin{definition}
\label{def:MtoSSYT}
Fix a partition $\lambda$ with frequency representation $[a_1,a_2,\ldots,a_{\lambda_1}]$ and fix $n\in\mathbb{N}$.
Let \emph{$M(\lambda,n)$} be the set of $\lambda_1\times n$ sign matrices $M=(M_{ij})$ such that:
\begin{align}
\label{eq:Mij_rowsum}
\displaystyle\sum_{j=1}^n M_{ij} &= a_{\lambda_1-i+1},
 & \mbox{ for all }1\le i \le \lambda_1.
\end{align}
Call $M(\lambda,n)$ the set of \emph{sign matrices of shape $\lambda$ and content at most $n$}. 
\end{definition}

\begin{theorem}
\label{thm:MtoSSYT}
$M(\lambda,n)$ is in explicit bijection with $SSYT(\lambda,n)$.
\end{theorem}

\begin{proof}
We first outline the bijection of Aval~\cite{aval}  between SSYT and sign matrices. Given an $m\times n$ sign matrix $M$, we construct a tableau $\Phi(M)=T \in SSYT(m,n)$ such that the entries in the $i$th row of $M$ determine the $(m-i+1)$st column (from the left) of $T$. In the $i$th row of $M$, note which columns have a partial sum (from the top) of one. Record the numbers of the matrix columns in which this occurs, in increasing order from top down, to form column $m-i+1$ of $T$. Since we record the entries in increasing order for each column of $T$ and each entry only occurs once in a column, the columns of $T$ are strictly increasing. The rows of $T$ are weakly increasing, since by (\ref{eq:sm2}) the partial sums of the rows of $M$ are non-negative. Thus, $T$ is a SSYT. The length of the first row of $T$ is $m$ and the entries of $T$ are at most $n$, since $M$ is an $m\times n$ matrix. Thus $\Phi$ maps into $SSYT(m,n)$. 

Aval proved in~\cite{aval} that $\Phi$ is an invertible map that gives a bijection between $SSYT(m,n)$ and $m\times n$ sign matrices. We refine this to a bijection between $SSYT(\lambda,n)$ and $M(\lambda,n)$ by keeping track of the row sums of $M$ and the shape of $T$. 
Given a tableau, $T\in SSYT(\lambda,n)$, we show that $\Phi^{-1}(T)=M\in M(\lambda,n)$. By~\cite{aval}, we know that $M$ is a sign matrix, so we only need to show it satisfies the condition  (\ref{eq:Mij_rowsum}). Consider the frequency representation $[a_1,a_2,a_3, \dots, a_{\lambda_1}]$ of the partition $\lambda$. Consider columns $\lambda_1-i$ and $\lambda_1-i+1$ of $T$.
If a number, $\ell$, appears in both columns $\lambda_1-i+1$ and $\lambda_1-i+2$ of $T$, then $M_{i\ell}=0$.
So we can ignore when a number is repeated in adjacent columns of $T$, since it corresponds to a zero in $M$, which does not contribute to the row sum. 
Suppose $\ell$ appears in column $\lambda_1-i+2$ of $T$ but not column $\lambda_1-i+1$.  Then $M_{i\ell}=-1$.
Suppose $\ell$ appears in column $\lambda_1-i+1$ of $T$ but not column $\lambda_1-i+2$.  Then $M_{i\ell}=1$.
So the total row sum $\displaystyle\sum_{j'=1}^n M_{ij'}$ equals the number of entries that appear in column $\lambda_1-i+1$ of $T$ but not column $\lambda_1-i+2$ minus the number of entries that appear in column $\lambda_1-i+2$ but not column $\lambda_1-i+1$. This is exactly the length of column $\lambda_1-i+1$ minus the length of column $\lambda_1-i+2$, which is given by $a_{\lambda_1-i+1}$.

See Figure~\ref{fig:(3,3,1,1,1)} and Example~\ref{ex:MtoSSYT}.
\end{proof}

\begin{example}
\label{ex:MtoSSYT}
In Figure~\ref{fig:(3,3,1,1,1)}, we have a semistandard Young tableau $T$ of shape $[3,3,1,1,1]$ and the corresponding sign matrix $M$ formed by the 
bijection discussed in
Theorem~\ref{thm:MtoSSYT}. 
To see that $M$ satisfies (\ref{eq:Mij_rowsum}), note that the total row sums of $M$ are $2$, $0$ and $3$, while the frequency representation of the partition $[3,3,1,1,1]$ is $[3,0,2]$. 
\end{example}

\begin{figure}[htbp]
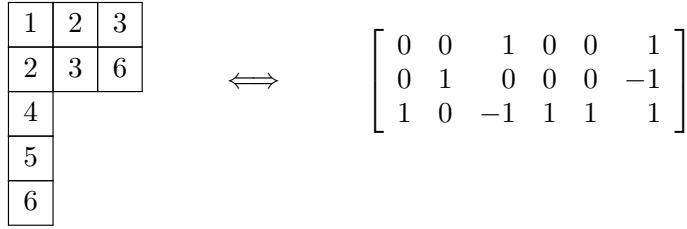

\begin{center}
$\raisebox{.22in}{
\begin{ytableau}
1&2&3\\ 2&3&6\\ 4\\5\\6
\end{ytableau}} \hspace{.4in} \Longleftrightarrow \hspace{.4in} \left[ \begin{array}{rrrrrr} 
0&0&1&0&0&1\\0&1&0&0&0&-1\\ 1&0&-1&1&1&1 \end{array} \right]$ 
\end{center}
\caption{The SSYT of shape $[3,3,1,1,1]$ and corresponding sign matrix from Example~\ref{ex:MtoSSYT}.}
\label{fig:(3,3,1,1,1)}
\end{figure}

\section{Definition and vertices of $P(\lambda,n)$}
\label{sec:Plambda_n}

In this section, we define the first of the two polytopes that we are studying and prove some of its properties. 

\begin{definition}
Let \emph{$P(\lambda,n)$} be the polytope defined as the convex hull, as vectors in $\mathbb{R}^{\lambda_1 n}$, of all the matrices in $M(\lambda,n)$. 
Call this the \emph{sign matrix polytope of shape $\lambda$.} \end{definition}

We now investigate the structure of this polytope, starting with its dimension.

\begin{proposition}
\label{prop:dim}
The dimension of $P(\lambda, n)$ is $\lambda_1(n-1)$ if $1 \leq k < n$. 
When $k=n$, the dimension is $(\lambda_1 - \lambda_n)(n-1).$ 
\end{proposition}

\begin{proof}
Since each matrix in $M(\lambda,n)$ is $\lambda_1 \times n$, the ambient dimension is $\lambda_1 n$.  
However, when constructing the sign matrix corresponding to a tableau of shape $\lambda$, as in Theorem~\ref{thm:MtoSSYT}, the last column is determined by the shape $\lambda$ via the prescribed row sums (\ref{eq:Mij_rowsum}) of Definition~\ref{def:MtoSSYT}. This is the only restriction on the dimension when 
$1 \leq k < n$, where $k$ is the length of $\lambda$, reducing the free entries in the matrix by one column. Thus,
the dimension is $\lambda_1(n-1)$. 

When $k=n$ the dimension depends on the number of columns of length $n$ in $\lambda$; this is given by $\lambda_n$. A column of length $n$ in a SSYT with entries at most $n$ is forced to be filled with the numbers $1,2,\ldots,n$. So the matrix rows corresponding to these columns are determined, and thus do not contribute to the dimension.
Thus the dimension is $(\lambda_1 - \lambda_n)(n-1)$. 
\end{proof}

From now on, we assume $k<n$. We now define a graph associated to any matrix. The graph will be useful in upcoming theorems; see Figure~\ref{fig:partialsums}.

\begin{definition}
\label{def:gamma}
We define the $m\times n$ \emph{grid graph} $\Gamma_{(m,n)}$ as follows.
The vertex set is $V(m,n):=\{(i,j) : 1 \leq i \leq m+1, 1 \leq j \leq n+1  \}$. 
We separate the vertices into two categories. We say the \emph{internal vertices} are \{$(i,j)$ : $1 \leq i \leq m, 1 \leq j \leq n$\} and the \emph{boundary vertices} are $\{(m+1,j) \mbox{ and } (i,n+1) : 1 \leq i \leq m, 1 \leq j \leq n\}$. The edge set is
\[E(m,n):=  \begin{cases} (i,j) \text{ to } (i+1,j) & 1 \leq i \leq m, 1 \leq j \leq n\\ (i,j) \text{ to } (i,j+1) & 1 \leq i \leq m, 1 \leq j \leq n. \end{cases}\]
We draw the graph with $i$ increasing to the right and $j$ increasing down, to correspond with matrix indexing. 
\end{definition}

\begin{definition}
\label{def:Xhat}
Given an $m\times n$ matrix $X$, we define a graph, $\thickhat{X}$, which is a labeling of the edges of $\Gamma_{(m,n)}$ from Definition~\ref{def:gamma}. 
The horizontal edges from $(i,j)$ to $(i,j+1)$ are each labeled by the corresponding row partial sum $r_{ij}= \displaystyle\sum_{j'=1}^{j} X_{ij'}$ ($1\leq i\leq m$, $1\leq j\leq n$). Likewise, the vertical edges from $(i,j)$ to $(i+1,j)$ are each labeled by the corresponding  column partial sum $c_{ij} = \displaystyle\sum_{i'=1}^{i} X_{i'j}$ ($1\leq i\leq m$, $1\leq j\leq n$).  In many of the figures, we will label the interior vertices with their corresponding matrix entry $\textcolor{blue}{\bf{X_{ij}}}$ ($1\leq i\leq m$, $1\leq j\leq n$). 
\end{definition}

\begin{remark}
\label{remark:invertible}
Note that given either the row or column partial sum labels of $\thickhat{X}$, one can uniquely recover the matrix $X$. 
\end{remark}

See Figures~\ref{fig:partialsums} and~\ref{fig:vertices}.

\begin{figure}[htbp]
\begin{center}
\includegraphics[scale=.65]{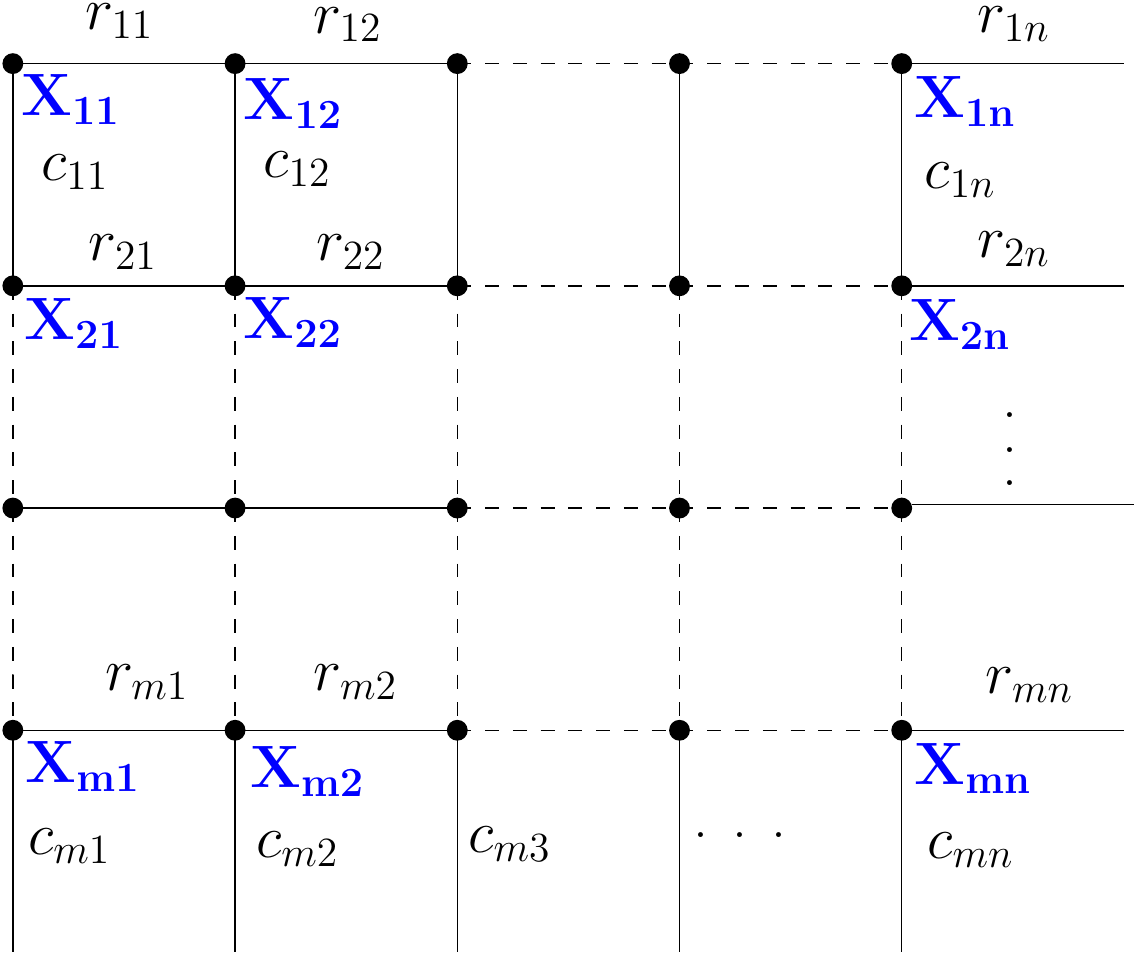}
\end{center}
\caption{The graph $\thickhat{X}$ from Definition~\ref{def:Xhat}, with dots on only the internal vertices. }
\label{fig:partialsums}
\end{figure}

The above notation will be used in proving the next theorem, which identifies the vertices of $P(\lambda,n)$. 

\begin{theorem}
\label{thm:lambdavertex}
The vertices of $P(\lambda,n)$ are the sign matrices $M(\lambda,n)$.
\end{theorem}

\begin{proof} Fix a sign matrix $M\in M(\lambda,n)$. In order to show that $M$ is a vertex of $P(\lambda,n)$, we need to find a hyperplane with $M$ on one side and all the other sign matrices in $M(\lambda,n)$ on the other side. Then since $P(\lambda,n)$ is the convex hull of $M(\lambda,n)$, $M$ will necessarily be a vertex.

Let $c_{ij}$ denote the column partial sums of $M$, as in Definition~\ref{def:Xhat}.
Define $C_M:=\{(i,j)\ :  \ c_{ij}=1\}$. 
Note that $C_M$ is unique for each $M$, since the column partial sums can only be $0$ or $1$, and by Remark~\ref{remark:invertible}, we can recover $M$ from the $c_{ij}$. Also note that $|C_M|=|\lambda|$, that is, the number of partial column sums that equal one in $M$ equals the number of boxes in $\lambda$. 

Define a hyperplane in $\mathbb{R}^{\lambda_1 n}$ as follows, on coordinates $X_{ij}$
corresponding to positions in a $\lambda_1\times n$ matrix.
\begin{equation} H_M(X):=\sum_{(i,j)\in C_M} \sum_{i'=1}^i X_{i'j}=|\lambda|-\frac{1}{2}
\label{eq:shape_hyper}
\end{equation}

If $X=M$, then $H_M(X)= H_M(M)=|\lambda|$, since $|C_M|=|\lambda|$.
Given a hyperplane formed in this manner, we may recover the matrix from which it is formed, thus $H_M$ is unique for each $M$. 

By definition, every matrix in $M(\lambda,n)$ has $|\lambda|$ partial column sums that equal $1$.  Let $M'\neq M$ be another matrix in $M(\lambda,n)$.
It must be that there is an $(i,j)$ where $c_{ij}=1$ in $M$ and $c_{ij}=0$ in $M'$. $H_{M}(M')$ will be smaller than $H_M(M)$ by one for every time this occurs.
For any $(i,j)$ such that $c_{ij}=0$ in $M$ and $c_{ij}=1$ in $M'$, $(i,j)\not\in C_M$, so this partial sum does not contribute to $H_M$. 

Therefore, 
$H_M(M)=|\lambda|>|\lambda|-\frac{1}{2}$ while $H_M(M')<|\lambda|-\frac{1}{2}$. 
Thus the sign matrices of $M(\lambda,n)$ are the vertices of $P(\lambda,n)$. 
\end{proof}

\begin{figure}[htbp]
\begin{center}
\includegraphics[scale=.8]{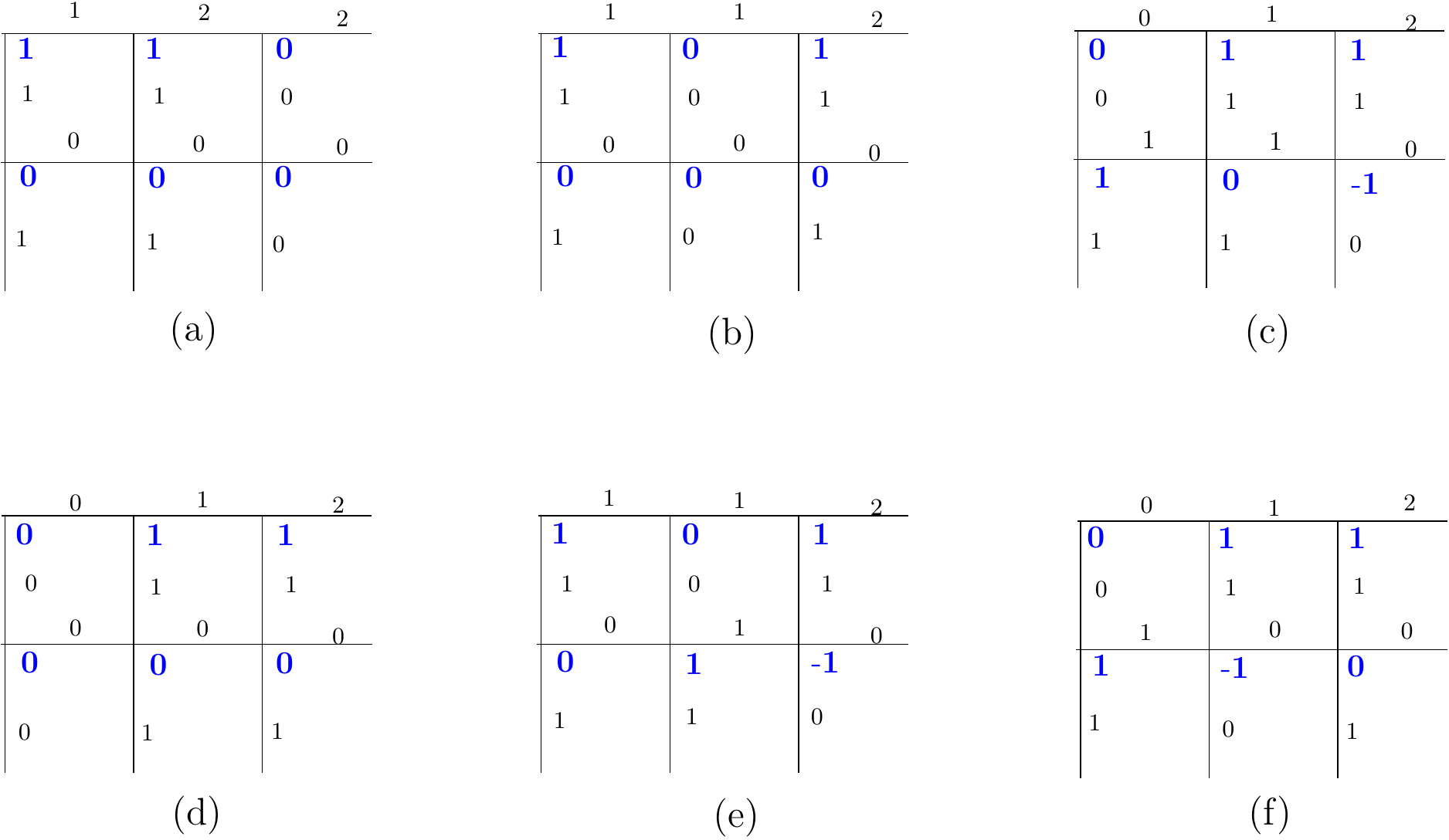}
\end{center}
\caption{The six graphs corresponding to the six sign matrices in $M([2,2],3)$; these matrices correspond to SSYT of shape $[2,2]$ with entries at most $3$.} 
\label{fig:vertices}
\end{figure}

\begin{example}
Figure~\ref{fig:vertices} gives the six graphs corresponding to the six sign matrices in $M(\lambda,3)$ for $\lambda=[2,2]$; these matrices correspond to SSYT of shape $[2,2]$ with entries at most $3$.
Let $M_e$ be the sign matrix corresponding to the graph in Figure~\ref{fig:vertices}(e). The equation for the hyperplane, $H_{M_e}$, described in Theorem~\ref{thm:lambdavertex},  is
$H_{M_e}(X)=X_{11}+(X_{11}+X_{21})+X_{13}+(X_{12}+X_{22})=2X_{11}+X_{12}+X_{13}+X_{21}+X_{22}=|\lambda|-\frac{1}{2}=3.5$. Now we substitute the entries of each matrix in $M([2,2],3)$ into this equation to show $M_e$ is the only matrix on one side of this hyperplane.

\tabbedblock{
(a): \>  $X_{11}=1, X_{12}=1, X_{13}=0, X_{21}=0, X_{22}=0$  \> $\rightarrow \;  H_{M_e}(M_a) =2+ 1 + 0 + 0 + 0$ \> =\;  3;\\ 
(b): \>  $X_{11}=1, X_{12}=0, X_{13}=1, X_{21}=0, X_{22}=0$  \> $\rightarrow \;   H_{M_e}(M_b)=2 + 0 + 1 + 0 + 0$ \> =\;  3;\\ 
(c): \>  $X_{11}=0, X_{12}=1, X_{13}=1, X_{21}=1, X_{22}=0$  \> $\rightarrow \;   H_{M_e}(M_c)=0 + 1 + 1 + 1 + 0$ \> =\;  3;\\ 
(d): \>  $X_{11}=0, X_{12}=1, X_{13}=1, X_{21}=0, X_{22}=0$  \> $\rightarrow \;   H_{M_e}(M_d)=0 + 1 + 1 + 0 + 0$ \> =\;  2;\\ 
(e): \> $X_{11}=1, X_{12}=0, X_{13}=1, X_{21}=0, X_{22}=1$  \> $\rightarrow \; H_{M_e}(M_e) = 2 + 0 + 1 + 0 + 1$ \> =\;  4;\\ 
(f): \> $X_{11}=0, X_{12}=1, X_{13}=1, X_{21}=1, X_{22}=\; $-1  \> $\rightarrow \; H_{M_e}(M_f) = 0 + 1 + 1 + 1 + $(-1) \> =\; 2.
}

Note that $M_e$ is on one side of $2X_{11}+X_{12}+X_{13}+X_{21}+X_{22}=3.5$ and the other five matrices in $M([2,2],3)$ are on the other side.
\label{ex:vertices}
\end{example}

\section{Definition and vertices of $P(m,n)$}
\label{sec:Pmn}

We will now define and study another family of polytopes, constructed using all $m \times n$ sign matrices.

\begin{definition}
Let $P(m,n)$ be the polytope defined as the convex hull of all $m\times n$ sign matrices. Call this the \emph{$(m,n)$ sign matrix polytope.}
\end{definition}

\begin{proposition}
\label{prop:pmn_dim}
The dimension of $P(m,n)$ is $mn$ for all $m > 1$.
\end{proposition}

\begin{proof}
Since every entry is essential, all $mn$ of the entries contribute to the dimension.
\end{proof}

\begin{theorem}
\label{thm:mnvertex}
The vertices of $P(m,n)$ are the sign matrices of size $m \times n$. 
\label{thm:mnverts}
\end{theorem}

\begin{proof}
Fix an $m \times n$ sign matrix $M$. In order to show that  $M$ is a vertex of $P(m,n)$, we need to find a hyperplane in $\mathbb{R}^{mn}$ with $M$ on one side and all the other $m\times n$ sign matrices on the other side. Then since $P(m,n)$ is the convex hull of all $m \times n$ sign matrices, $M$ would necessarily be a vertex.

Let $c_{ij}=\displaystyle\sum_{i'=1}^i X_{i'j}$ in $M$, as in Definition~\ref{def:Xhat}. 
Recall from the proof of Theorem~\ref{thm:lambdavertex} the notation $C_M=\{(i,j)\ :  \ c_{ij}=1\mbox{ in }M\}$ and $H_M(X)=\displaystyle\sum_{(i,j)\in C_M} \displaystyle\sum_{i'=1}^i X_{i'j}$.

Define a hyperplane in $\mathbb{R}^{mn}$ as follows, on coordinates $X_{ij}$ corresponding to positions in an $m\times n$ matrix.
\begin{equation}
K_M(X):=
H_M(X) - \sum_{(i,j)\not\in C_M} \sum_{i'=1}^i X_{i'j} = |C_M|-\frac{1}{2}.
\label{eq:mn_hyper}
\end{equation}

Note that $C_M$ is unique for each sign matrix $M$ since we may recover any sign matrix from its column partial sums (see Remark~\ref{remark:invertible}). Therefore $K_M$ is unique for each matrix $M$. 

We wish to show the hyperplane $K_M(X)=|C_M|-\frac{1}{2}$ has $M$ on one side and all the other $m\times n$ sign matrices on the other. 
Note that if $X=M$, then $K_M(X)= K_M(M)=|C_M|$. So we wish to show that given any $M'\in M(m,n)$ such that $M'\neq M$, $K_M(M')<|C_M|-\frac{1}{2}$.

We have two cases: 

\textit{Case 1}: There is a $(i,j)$ entry $c_{ij}=0$ in $M$ and $c_{ij}=1$ in $M'$. In this case, $(i,j)\not\in C_M$. So in $K_M(M')$, this partial sum gets subtracted making $K_M(M')$ one smaller than $K_M(M)$ for every such $(i,j)$.

\textit{Case 2}: There is a $(i,j)$ entry $c_{ij}=1$ in $M$ and $c_{ij}=0$ in $M'$. In this case, $(i,j)\in C_M$. So this partial sum contributed one to $H_M(M)$, whereas in $H_M(M')$ there is a contribution of zero. Therefore $H_M(M)$ is one greater than $H_M(M')$ so that $K_M(M)$ is one greater than $K_M(M')$ for every such $(i,j)$.

Since $M$ and $M'$ must differ in at least one column partial sum, $|C_M|=K_M(M)\geq K_M(M')+1$ so that $K_M(M')<|C_M|-\frac{1}{2}$ for all $m \times n$ sign matrices $M'$.
Thus the $m\times n$ sign matrices are the vertices of $P(m,n)$. 
\end{proof}

\begin{figure}[htbp]
\begin{center}
\includegraphics[scale=.8]{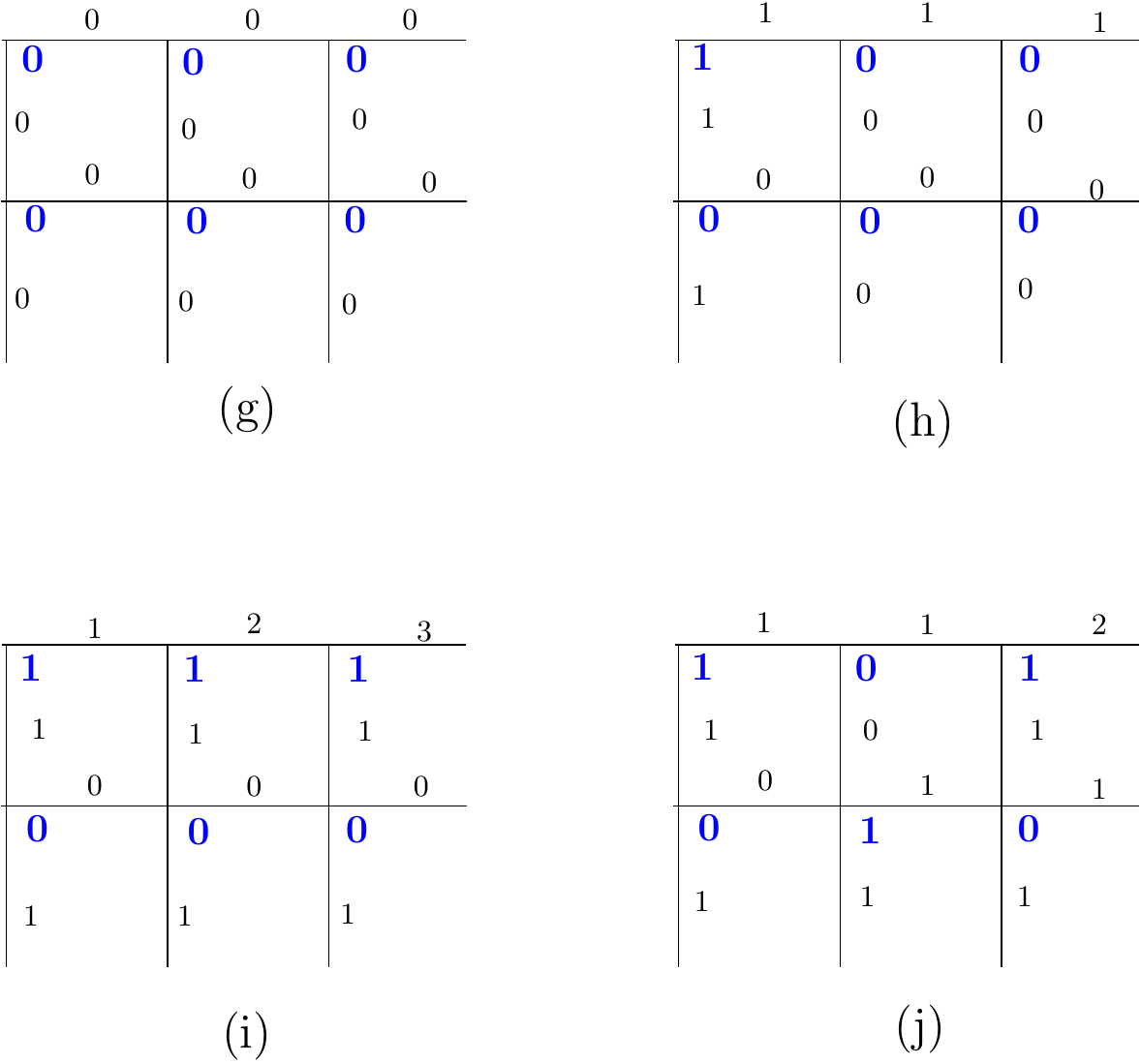}
\end{center}
\caption{Four of the $29$ partial sum graphs corresponding to the sign matrices that are vertices in $P(2,3)$ but not in $P([2,2],3)$.}
\label{fig:mnvertices}
\end{figure}

\begin{example}
\label{ex:mnverts}
Let $M_h$ be the sign matrix corresponding to the graph in Figure~\ref{fig:mnvertices}(h). So $H_{M_h}(X)=2X_{11}+X_{21}$, and therefore $H_{M_h}(M_a)=H_{M_h}(M_b)=H_{M_h}(M_e)=H_{M_h}(M_h)=H_{M_h}(M_i)=H_{M_h}(M_j)=2$. This shows that the hyperplane of Theorem~\ref{thm:lambdavertex} does not separate $M_h$ from all the other $m\times n$ sign matrices. But using Theorem~\ref{thm:mnverts}, we find the needed hyperplane to be $K_{M_h}(X)=X_{11}+(X_{11}+X_{21})-X_{12}-(X_{12}+X_{22})-X_{13}-(X_{13}+X_{23})=2X_{11}+X_{21}-2X_{12}-X_{22}-2X_{13}-X_{23}=|C_M|-\frac{1}{2}=2-\frac{1}{2}=1.5$. One may calculate the following:
$K_{M_h}(M_a) = K_{M_h}(M_b) = K_{M_h}(M_e) = 0; \; K_{M_h}(M_h) = 2; \; K_{M_h}(M_i)=-2; \; K_{M_h}(M_j)= -1$. This illustrates how the hyperplane $K_{M}(X)=|C_M|-\frac{1}{2}$  separates $M$ from the other $m\times n$ sign matrices, even though $H_{M}(X)=|C_M|-\frac{1}{2}$ fails to.
\end{example}

In the following remark, we give some properties and non-properties of  $P(m,n)$ and $P(\lambda,n)$. 

\begin{remark}
Both $P(\lambda,n)$ and $P(m,n)$ are \emph{integral polytopes}, since an integral polytope has integer values for all vertices. 
Neither $P(\lambda,n)$ nor $P(m,n)$ are \emph{regular polytopes}. 
(A regular polytope has the same number of edges adjacent to each vertex.)
For example, some of the vertices in $P(2,2)$ from Figure~\ref{fig:faces} are adjacent to 4 edges, while others are adjacent to 5 or 6 edges. 
These polytopes are not \emph{simplicial} (where every facet has the minimal number of vertices), since the facets of these polytopes have varying numbers of vertices. For example, the facets of $P(2,2)$ have between $4$ and $7$ vertices. These polytopes are not \emph{simple} (where every vertex is contained in the minimal number of facets where that number is fixed); the vertices corresponding to $\delta_5$ and $\delta_6$ in Figure~\ref{fig:faces} are contained in $20$ and $14$ facets, respectively.
\end{remark}

\section{Inequality descriptions}
\label{sec:ineq}
In analogy with the Birkhoff polytope \cite{birkhoff,vonneumann} and the alternating sign matrix polytope \cite{behrend,striker}, we find an inequality description of $P(\lambda,n)$. 

\begin{theorem}
\label{thm:ineqthmshape}
$P(\lambda,n)$ consists of all 
$\lambda_1\times n$ real matrices $X=(X_{ij})$ such that:
\begin{align}
\label{eq:eq1}
0 \leq \displaystyle\sum_{i'=1}^{i} X_{i'j} &\leq 1, &\mbox{ for all }1 \leq i\leq \lambda_1, 1\le j\le n \\
\label{eq:eq2}
0 \leq \displaystyle\sum_{j'=1}^{j} X_{ij'}, &
&\mbox{ for all }1 \leq i\leq \lambda_1, 1\le j\le n \\
\label{eq:eq3}
\displaystyle\sum_{j'=1}^n X_{ij'} &= a_{\lambda_1-i+1}, &\mbox{ for all } 1\leq i \leq \lambda_1.
\end{align}
\end{theorem}

\begin{proof}
This proof builds on techniques developed by Von Neumann in his proof of the inequality description of the Birkhoff polytope \cite{vonneumann}. First we need to show that any $X\in P(\lambda,n)$ satisfies $(\ref{eq:eq1}) - (\ref{eq:eq3})$. Suppose $X \in P(\lambda,n)$. Thus $X=\displaystyle \sum_{\gamma} \mu_{\gamma} M_{\gamma}$ where $\displaystyle \sum_{\gamma}  \mu_{\gamma}=1$ and the $M_{\gamma}\in M(\lambda,n)$. Since we have a convex combination of sign matrices, by Definition~\ref{def:sm} we obtain (\ref{eq:eq1}) and (\ref{eq:eq2}) immediately.  (\ref{eq:eq3}) follows from (\ref{eq:Mij_rowsum}) in the definition of $M(\lambda,n)$ (Definition~\ref{def:MtoSSYT}). Thus $P(\lambda,n)$ fits the inequality description. 

Let $X$ be a real-valued $\lambda_1\times n$ matrix satisfying $(\ref{eq:eq1}) - (\ref{eq:eq3})$.
We wish to show that $X$ can be written as a convex combination of  sign matrices in $M(\lambda,n)$, so that it is in $P(\lambda,n)$. 
Consider the corresponding graph $\thickhat{X}$ of
Definition~\ref{def:Xhat}. 
Let $r_{i0}=0=c_{0j}$ for all $i,j$. Then for all $1\leq i\leq \lambda_1, 1\leq j\leq n$, we have $X_{ij}=r_{ij}-r_{i, j-1}=c_{ij}-c_{i-1,j}$.  Thus,
\begin{equation}
\label{eq:rc}
r_{ij}+c_{i-1,j}=c_{ij}+r_{i, j-1}.
\end{equation}
If $X$ has no non-integer partial sums, then  $X$ is a $\lambda_1 \times n$ sign matrix, since $(\ref{eq:eq1}) - (\ref{eq:eq3})$ reduce to Definitions \ref{def:sm} and \ref{def:MtoSSYT}.

So we assume $X$ has at least one non-integer partial sum $r_{ij}$ or $c_{ij}$. We may furthermore assume $X$ has at least one non-integer column partial sum, since if all column partial sums of $X$ were integers, $X_{ij}=c_{ij}-c_{i-1,j}$ would imply the $X_{ij}$ would be integers, thus all row partial sums would also be integers. 

We construct an \emph{open} or \emph{closed circuit} 
in $\thickhat{X}$ whose edges are labeled by non-integer partial sums. 
We say a \emph{closed circuit} is a simple cycle in $\thickhat{X}$, that is, it begins and ends at the same vertex with no repetitions of vertices, other than the repetition of the starting and ending vertex. We say an \emph{open circuit} is a simple path in $\thickhat{X}$ that begins and ends at different boundary vertices  along the bottom of the graph, that is, it begins at a vertex $(\lambda_1+1,j)$ and ends at vertex $(\lambda_1+1,j_0)$ for some $j_0 \neq j$. 

We create such a circuit by first constructing a path in $\hat{X}$ as follows.
If there exists $j$ such that $0<c_{\lambda_1j}<1$, we start the path at bottom boundary vertex $(\lambda_1+1,j)$. 
If there is no such $j$, we find some $c_{ij}$ such that $0<c_{ij}<1$ and start at the vertex corresponding to $X_{ij}$. 
By (\ref{eq:rc}), at least one of $c_{i\pm 1,j},r_{i,j\pm 1}$ is also a non-integer. 
Therefore, we may form a path by moving through $\thickhat{X}$ vertically and horizontally along edges labeled by non-integer partial sums. 

Now $\hat{X}$ is of finite size and all the boundary partial sums on the left, right, and top are integers (since for all $i$ and $j$, $r_{i0}=c_{0j}=0$ and  $r_{in}=a_{\lambda_1-i+1}$). So the path eventually reaches one of the following: $(1)$ a vertex already in the path,
or $(2)$ a vertex $(\lambda_1+1,j_0)$. In Case $(2)$, this means $c_{\lambda_1j_0}$ is not an integer. But the total sum of the matrix is $\displaystyle\sum_{i=1}^{\lambda_1} r_{in}=\displaystyle\sum_{i=1}^{\lambda_1} a_{\lambda_1-i+1}$. Each $a_{\lambda_1-i+1}$ is an integer, so the total sum of all matrix entries is an integer. Since $c_{\lambda_1j_0}$ is not an integer, there must be some other column sum $c_{\lambda_1j}$ that is also not an integer. By construction, the path began at a bottom boundary vertex $(\lambda_1+1,j)$ with $c_{\lambda_1 j}$ not an integer, for some $j\neq j_0$. So this process yields an open circuit whose edge labels are all non-integer. In Case $(1)$, the constructed path consists of a simple closed loop and possibly a simple path connected to the closed loop at some vertex $X_{i_0 j_0}$. We delete this path, and keep the closed loop. This process yields a closed circuit in $\thickhat{X}$ whose edge labels are all non-integer. 
See Figures~\ref{fig:opencirc} and~\ref{fig:closedcirc} for examples.

Let the following denote a circuit constructed as above, where the circled $c$ and $r$ values denote the edge labels as we traverse the circuit, and 
 the boxed $X_{ij}$'s denote the matrix entries corresponding to the vertices on the corners of the circuit where the path changes from vertical to horizontal or vice versa. (Note how the boxes and circles appear in Figures~\ref{fig:opencirc} and~\ref{fig:closedcirc}.) 

\[\left(
\circled{$c_{0}$},\ldots,\circled{$c'_{0}$},\boxed{X_{i_1,j_0}},\circled{$r_{1}$},\ldots,\circled{$r'_{1}$},\boxed{X_{i_1,j_1}},\circled{$c_{1}$},\ldots,\circled{$c'_{1}$},\boxed{X_{i_2,j_1}},\circled{$r_{2}$},\ldots\right)\]
Using this circuit, we are able to write $X$ as the convex combination of two new matrices, call them $X^+$ and $X^-$, that each have at least one more partial sum equal to its maximum or minimum possible value. 

Construct a matrix $X^+$ by setting 
\[X^+_{i_{\alpha},j_{\beta}} =
\begin{cases}
X_{i_{\alpha}j_{\beta}} + \ell^+ &\mbox{ if } \alpha+\beta \mbox{ is odd} \\
X_{i_{\alpha}j_{\beta}} - \ell^+ &\mbox{ if } \alpha+\beta \mbox{ is even} 
\end{cases}\]
and setting all other entries equal to the corresponding entry of $X$.
That is, construct $X^+$ by alternately adding and subtracting a number $\ell^+$ from each entry in $X$ that corresponds to a corner in the circuit and leaving all other matrix entries unchanged. 
We will choose $\ell^+$ to be the maximum possible value that preserves $(\ref{eq:eq1}) - (\ref{eq:eq3})$ when added and subtracted from the corners as indicated above. 
That is, $\ell^+$ equals the minimum value of the union of the following sets: 
\begin{align*}
\{&c_{ij} \ | \ \mbox{ the edge labeled by } c_{ij} \mbox{ is below a circuit corner } X_{i_{\alpha}j_{\beta}} \mbox{ with } \alpha+\beta \mbox{ even}\},\\
\{&1-c_{ij} \ | \ \mbox{ the edge labeled by } c_{ij} \mbox{ is below a circuit corner } X_{i_{\alpha}j_{\beta}} \mbox{ with } \alpha+\beta \mbox{ odd}\},\\
\{&r_{ij} \ | \ \mbox{ the edge labeled by } r_{ij} \mbox{ is to the right of a circuit corner } X_{i_{\alpha}j_{\beta}} \mbox{ with } \alpha+\beta \mbox{ even}\}.
\end{align*}
Note $\ell^+>0$ since all the partial sums in the circuit are non-integer. 

Construct a matrix $X^-$ by setting 
\[X^-_{i_{\alpha},j_{\beta}} =
\begin{cases}
X_{i_{\alpha}j_{\beta}} - \ell^- &\mbox{ if } \alpha+\beta \mbox{ is odd} \\
X_{i_{\alpha}j_{\beta}} + \ell^- &\mbox{ if } \alpha+\beta \mbox{ is even.} 
\end{cases}\]
and setting all other entries equal to the corresponding entry of $X$.
That is, construct $X^-$ by alternately subtracting and adding a number $\ell^-$ from each entry in $X$ that corresponds to a corner in the circuit and leaving all other matrix entries unchanged. 
We will choose $\ell^-$ to be the maximum possible value that preserves (\ref{eq:eq1}), (\ref{eq:eq2}), and (\ref{eq:eq3}) when subtracted and added from the corners as indicated above. That is, $\ell^-$ equals the minimum value of the union of the following sets: 
\begin{align*}\{&c_{ij} \ | \ \mbox{ the edge labeled by } c_{ij} \mbox{ is below a circuit corner } X_{i_{\alpha}j_{\beta}} \mbox{ with } \alpha+\beta \mbox{ odd}\},\\
\{&1-c_{ij} \ | \ \mbox{ the edge labeled by } c_{ij} \mbox{ is below a circuit corner } X_{i_{\alpha}j_{\beta}} \mbox{ with } \alpha+\beta \mbox{ even}\},\\
\{&r_{ij} \ | \ \mbox{ the edge labeled by } r_{ij} \mbox{ is to the right of a circuit corner } X_{i_{\alpha}j_{\beta}} \mbox{ with } \alpha+\beta \mbox{ odd}\}.
\end{align*}
Note $\ell^->0$ since all the partial sums in the circuit are non-integer. 

Now in the case of either an open or closed circuit, there will be an even number of corners 
in the circuit. Note that for open circuits, each row has an even number of corners and there will be two columns with an odd number of corners, namely the columns where the path begins and ends. Whenever there is an even number of circuit corners in a row or column, this means that the same number is alternately added to and subtracted from the corners, thus the total row or column sum is not changed. Whenever there is an odd number of circuit corners in a column, this means that the total column sum will change, however it will stay between $0$ and $1$. Thus our constructions of $X^+$ and $X^-$ above are well-defined. 

Both $X^+$ and $X^-$ satisfy (\ref{eq:eq1})--(\ref{eq:eq3}) by construction.
Also by construction, 
\[X=\frac{\ell^-}{\ell^++\ell^-}X^++\frac{\ell^+}{\ell^++\ell^-}X^-\] and $\frac{\ell^-}{\ell^++\ell^-} + \frac{\ell^+}{\ell^++\ell^-} = 1$. So $X$ is a convex combination of the two matrices $X^+$ and $X^-$ that still satisfy the inequalities and are each at least one step closer to being sign matrices, since they each have at least one more partial sum attaining its maximum or minimum bound.
Hence, by iterating this process, $X$ can be written as a convex combination of sign matrices in $M(\lambda,n)$.
\end{proof}

\begin{figure}[htbp]
\begin{multicols}{2}

\noindent \vspace{.5in}

$\left[\begin{array}{rrrr} 
.9 & 0 & .3 & .8  \\ 
0 & .1 & .6 & -.7  \\ 
0 & .9 & -.1 & .2  \\ 
\end{array}\right] \hspace{.5in} \implies $

\vspace{.2in}

 \includegraphics[scale=.7]{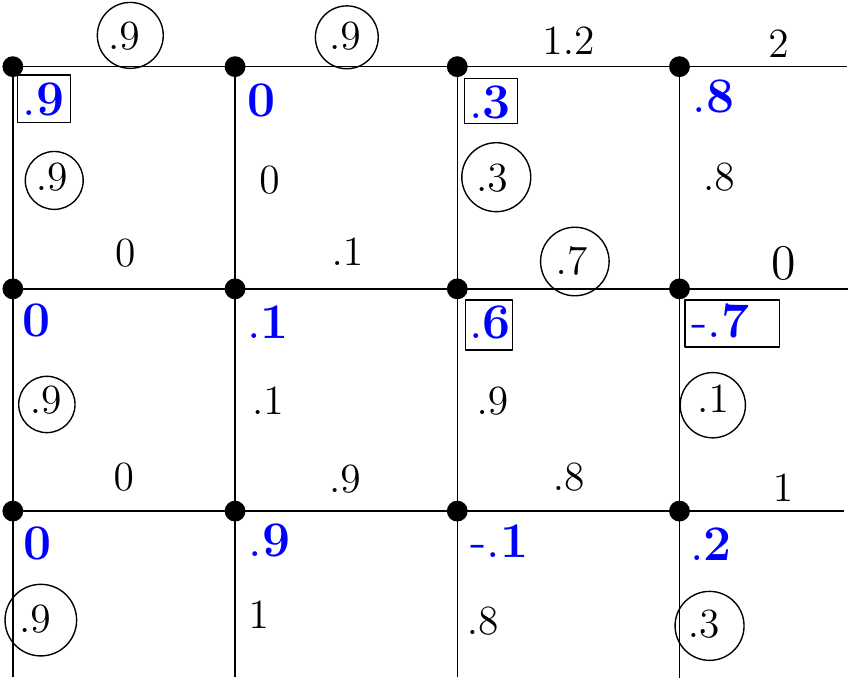}
\end{multicols}

\caption{Left: A matrix $X$ in $P([3,3,1],4)$; Right: An open circuit in $\thickhat{X}$.}
\label{fig:opencirc}
\end{figure}

\begin{figure}[htbp]
\begin{multicols}{2}

\noindent \vspace{.5in}

$\left[\begin{array}{rrrr} 
1 & 0 & 0 & 1  \\ 
0 & .4 & .6 & -1  \\ 
0 & .6 & -.6 & 0  \\ 
\end{array}\right] \hspace{.5in} \implies$

\vspace{.2in}

\includegraphics[scale=.7]{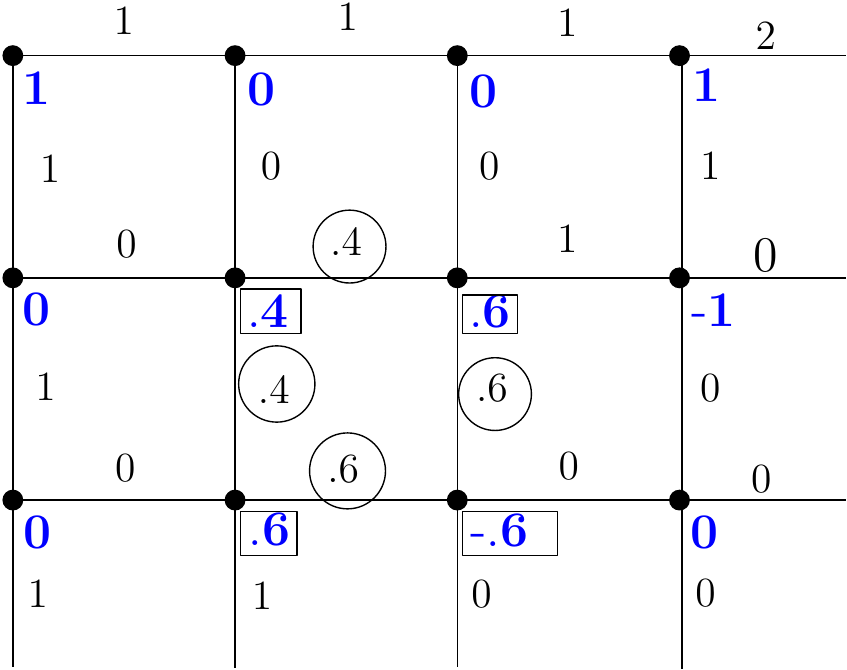} 
\end{multicols}

\caption{Left: A matrix $X$ in $P([3,3],4)$; Right: A closed circuit in $\thickhat{X}$.}
\label{fig:closedcirc}
\end{figure}

\begin{figure}[htbp]
\[\left[\begin{array}{rrrr} 
.9 & 0 & .3 & .8  \\ 
0 & .1 & .6 & -.7  \\ 
0 & .9 & -.1 & .2  \\ 
\end{array}\right] = \frac{.7}{.1+.7} \left[\begin{array}{rrrr} 
1 & 0 & .2 & .8  \\ 
0 & .1 & .7 & -.8  \\ 
0 & .9 & -.1 & .2  \\ 
\end{array}\right] + \frac{.1}{.1+.7} 
\left[\begin{array}{rrrr} .2 & 0 & 1 & .8  \\ 
0 & .1 & -.1 & 0  \\ 
0 & .9 & -.1 & .2  \\
\end{array}\right]\]
\caption{The decomposition of the matrix from Figure~\ref{fig:opencirc} as the convex combination of $X^+$ and $X^-$; see Example~\ref{ex:Xk}.}
\label{fig:Xk}
\end{figure}

\begin{example}
We use the open circuit in Figure~\ref{fig:opencirc}, we will find $X^+, X^-, \ell^+$ and $\ell^-$. The circuit is $\left(\textbf{\circled{.9}},\; \textbf{\circled{.9}},\; \textbf{\circled{.9}},\;\boxed{\color{blue} \textbf{.9}}, \;  \circled{.9},\; \circled{.9}, \;\boxed{\color{blue} \textbf{.3}},\; \textbf{\circled{.3}}, \;\boxed{\color{blue} \textbf{.6}},\;  \circled{.7}, \;\boxed{\color{blue} \textbf{-.7}},\;  \textbf{\circled{.1}},\; \textbf{\circled{.3}}\right)$, where the circled and bold entries are the partial column sums and the circled non-bold entries are the row partial sums of the circuit. The matrix entries at the corners of the circuit are boxed for emphasis. To construct $X^+$, we label the corner entries alternately plus and minus, so the plus value goes on the $\boxed{\color{blue} \textbf{.9}}$ and $\boxed{\color{blue} \textbf{.6}}$ corners and the minus on the $\boxed{\color{blue} \textbf{.3}}$ and $\boxed{\color{blue} \textbf{-.7}}$ corners. 
Looking at the partial sums, we see that $\ell^+$ will be the minimum of $\left\{.3,\; .1, \; .3 \right\}\cup\left\{1-.9,\; 1-.9,\; 1-.9\right\}\cup\emptyset$. Thus $\ell^+ =.1$, so $.1$ will be added to plus corners and subtracted from minus corners with $X^+$ as the result. We now switch the plus and minus corners. $\ell^-$ will be the minimum of $\left\{.9,\; .9,\; .9\right\}\cup\left\{1-.3,\; 1-.1, \; 1-.3 \right\}\cup\left\{.9,\; .9,\; .7 \right\}$
so $\ell^- = .7$. So then $.7$ is added to the plus corners and subtracted from the minus corners to get $X^-$. Thus we may write the matrix as the convex combination of the matrices $X^+$ and $X^-$ as in Figure~\ref{fig:Xk}.
\label{ex:Xk}
\end{example}

We now find an inequality description of $P(m,n)$.

\begin{theorem}
\label{thm:ineqthm}
$P(m,n)$ consists of all $m\times n$ real matrices $X=\{X_{ij}\}$ such that:
\begin{align}
0 &\le \displaystyle\sum_{i'=1}^{i} X_{i'j} \le 1 &\mbox{ for all } 1\le i \le m, 1\le j\le n. \\
0 &\le \displaystyle\sum_{j'=1}^{j} X_{ij'} 
&\mbox{ for all }1\le i \le m, 1 \le j \le n. 
\end{align}
\end{theorem}

\begin{proof}
The proof follows the proof of Theorem~\ref{thm:ineqthmshape}, with a few differences. The open circuits are no longer restricted to start and end at the bottom of the matrix; they may also start and end at vertices $(i,n+1)$ and $(i_0,n+1)$ ($i\neq i_0$) on the right border of $\Gamma_{(m,n)}$,
 or they may start at the bottom at vertex  $({m+1,j})$ and end on the right at vertex $(i,n+1)$. Therefore the evenness of corners is not needed here, since unlike in Theorem~\ref{thm:ineqthmshape}, there is no analogue of Equation~(\ref{eq:eq3}) that specifies the row sums. With these less restrictive exceptions, the matrices $X^-$ and $X^+$ will be found in the same way as in the proof of Theorem~\ref{thm:ineqthmshape}.
\end{proof}

\section{Facet enumerations}
\label{sec:Pmninq}
In this section, we use the inequality descriptions of the previous section to enumerate the facets in $P(m,n)$ and $P(\lambda,n)$. Note this is not as straightforward as counting the inequalities in the theorems of the previous section, as these inequality descriptions are not minimal.

\begin{theorem}
\label{thm:facets_mn}
$P(m,n)$ has 
$3mn-n-2(m-1)$ facets.
\end{theorem}

\begin{proof}
We have three defining inequalities in the inequality description of Theorem~\ref{thm:ineqthm} for each entry $X_{ij}$ of $X\in P(m,n)$: $0 \leq \displaystyle\sum_{i'=1}^{i} X_{i'j}$, $\displaystyle\sum_{i=1}^{i} X_{i'j} \leq 1$, and $0 \leq \displaystyle\sum_{j'=1}^{j} X_{ij'}$. Therefore there are at most $3mn$ facets, each made by turning one of the inequalities to an equality. We now determine which of these inequalities give \emph{unique} facets. (See Figure~\ref{fig:facet_mn} for a visual representation of which inequalities determine duplicate facets.)

Notice first that we will always have $0 \leq X_{1j}$ (from the column partial sums). This implies that the partial sums of the first row are all nonnegative, since each entry in the first row must be nonnegative.
Thus the inequalities $0 \leq \displaystyle\sum_{j=1}^{j} X_{1j'}$ for  $1\leq j\leq n$ are all unnecessary; and there are $n$ inequalities of this form.

We have already counted $0 \leq X_{11}$ in the column partial sums. From the partial row sums, we have that $0 \leq X_{21}$. But in the partial column sum we have $0 \leq X_{11}+X_{12}$; this is implied by $0 \leq X_{11}$ and $0 \leq X_{21}$. 
Similarly, the partial column sums $0 \leq \displaystyle\sum_{i'=1}^{i} X_{i'1}$ for $2\leq i\leq m$ are all implied by the partial row sums $0 \leq X_{i'1}$. There are $m-1$ inequalities of this form.

Note that $\displaystyle\sum_{i'=1}^{m} X_{i'1} \leq 1$. Furthermore, note that $0 \leq X_{m1}$ from the row partial sums. Therefore we have that $\displaystyle\sum_{i'=1}^{m-1} X_{i'1 }\leq 1 - X_{m1} \leq 1$. Similarly, the $m-1$ inequalities in the form of  $\displaystyle\sum_{i'=1}^{i} X_{i'1} \leq 1$ for $1\leq i < m$ are all implied by the partial row sums $0 \leq X_{i'1}$.

Therefore we have the number of facets to be at most $3mn-n-2(m-1)$. We claim this upper bound is the facet count. That is, a facet can be defined as all ${X} \in P(m,n)$ which satisfy exactly one of the following:
\begin{align}
\label{eq:f1}
r_{ij} &= \displaystyle\sum_{j'=1}^{j} X_{ij'} = 0, \hspace{.3in} & 2 \leq i \leq m \text{ and }
1 \leq j \leq n \\
\label{eq:f2}
c_{ij} &= \displaystyle\sum_{i'=1}^{i} X_{i'j} =0, \hspace{.3in} & 1 \leq i \leq m \text{ and } 2 \leq j \leq n 
\\
\label{eq:f3}
c_{ij} &=\displaystyle\sum_{i'=1}^{i} X_{i'j}=1, \hspace{.3in} & 1 \leq i \leq m \text{ and } 2 \leq j \leq n \\
\label{eq:f4}
r_{11} &= c_{11} = X_{11} =0 \\
\label{eq:f5}
c_{m1} &= \displaystyle\sum_{i'=1}^{m} X_{i'1}=1.
\end{align}

Note each equality fixes exactly one entry, thus lowering the dimension by one. 
Let two generic equalities of the  form (\ref{eq:f1})-(\ref{eq:f5}) be denoted as
$\alpha_{ij}=\gamma$ 
and $\beta_{de}=\delta$ for $\alpha,\beta\in\{r,c\}$ and $\gamma,\delta\in\{0,1\}$, where the choice of $r$ or $c$ for each of $\alpha$ and $\beta$ indicates whether the equality involves a row partial sum $r_{ij}$ or column partial sum $c_{ij}$, and the indices $(i,j)$ and $(d,e)$ must be in the corresponding ranges indicated by (\ref{eq:f1})-(\ref{eq:f5}). To finish the proof, we construct an $m\times n$ sign matrix $M$, such that $M$ satisfies $\alpha_{ij}=\gamma$ and not $\beta_{de}=\delta$. 
We work with $\thickhat{M}$ rather than $M$ itself, recalling the bijection between $M$ and $\thickhat{M}$. Recall from Definition~\ref{def:Xhat}, $\thickhat{M}$ is a graph whose horizontal edges are labeled by the partial row sums of $M$ and whose vertical edges are labeled by the partial column sums of $M$. Since all of the equalities in (\ref{eq:f1})-(\ref{eq:f5}) are given by setting a $c_{ij}$ equal to 0 or 1 or a $r_{ij}$ equal to 0, set the edge label of $\thickhat{M}$ corresponding to $\alpha_{ij}$ equal to $\gamma$ and the edge label corresponding to the equality $\beta_{de}$ equal to $1-\delta$. Now we transform $\thickhat{M}$ back to $M$ and if we can fill in the rest of the matrix so it is a sign matrix, the proof will be complete. In the cases below, we construct such a sign matrix $M$ satisfying equality $\alpha$ and not equality $\beta$.

\smallskip
\emph{Case 1}: $\alpha_{ij}=0$ and $\beta_{de}=1$. So in $\thickhat{M}$,  $\beta_{de}=0$. It suffices to set $M$ equal to the zero matrix.

\smallskip
\emph{Case 2:} $\alpha_{ij}=0$ and $\beta_{de}=0$. So in $\thickhat{M}$,  $\beta_{de}=1$. If $i\neq d$ and $j\neq e$, let $M_{de}=1$ and the rest of the entries equal to zero. 

Suppose $\alpha=\beta=c$. If $j\neq e$, let $M_{de}=1$ and the rest of the entries equal to zero. If $j=e$ and $i<d$,  let $M_{de}=1$ and the rest of the entries equal to zero. If $j=e$ and $i>d$,  let $M_{de}=1$, $M_{d+1,e}=-1$,  $M_{d+1,e-1}=1$, and the rest of the entries equal to zero. (Note $e\geq 2$ since $\beta=c$.)

Suppose $\alpha=\beta=r$. If $i\neq d$, let $M_{de}=1$ and the rest of the entries equal to zero. If $i=d$ and $j<e$, let $M_{de}=1$ and the rest of the entries equal to zero. If $i=d$ and $j>e$, let $M_{de}=1$, $M_{d,e+1}=-1$,  $M_{d-1,e+1}=1$, 
and the rest of the entries equal to zero. Note since $\beta=r$, $d\geq 2$, so $d-1\geq 1$. 

If $\alpha=r$ and $\beta=c$, let $M_{1e}=1$ and the rest of the entries equal to zero. (Note since $\alpha=r$, $i\geq 2$.)

If $\alpha=c$ and $\beta=r$, let $M_{d1}=1$ and the rest of the entries equal to zero.  (Note since $\alpha=c$, $j\geq 2$.)

\smallskip
\emph{Case 3:} $\alpha_{ij}=1$ and $\beta_{de}=1$. So in $\thickhat{M}$,  $\beta_{de}=0$. Note only column partial sums are set equal to 1 in the above list of equalities, so $\alpha=c$ and $\beta=c$. If $j\neq e$, set $M_{ij}=1$ and the rest of the entries of $M$ equal to zero. If $j=e$ and $i<d$, set $M_{ij}=M_{i+1,j-1}=1$ and $M_{i+1,j}=-1$ and all other entries equal to zero. Note $j-1\geq 1$ since (\ref{eq:f3}) requires that $2\leq j\leq n$.
If $j=e$ and $i>d$, set $M_{ij}=1$ and the rest of the entries of $M$ equal to zero.

\smallskip
\emph{Case 4:} $\alpha_{ij}=1$ and $\beta_{de}=0$. So in $\thickhat{M}$,  $\beta_{de}=1$. Note $\alpha=c$, so $j\geq 2$. 
If $j\neq e$, let $M_{ij}=M_{de}=1$ and the rest of the entries zero. If $j=e$ and $\beta=c$, let $M_{1j}=1$ and the rest of the entries equal to zero. If $j=e$ and $\beta=r$, if $i\neq d$, let $M_{ij}=1$ and $M_{d1}=1$ (we noted above that $j\geq 2$, so these ones are not in the same column) and the rest of the entries equal to zero. If $j=e$, $\beta=r$, and $i=d$, set $M_{ij}=1$ and the rest of the entries equal to zero.

\smallskip
Thus we may always complete to a sign matrix. 
$M$ is constructed to satisfy $\alpha_{ij}=\gamma$ but not $\beta_{ij}=\delta$, thus each of the equalities in (\ref{eq:f1})-(\ref{eq:f5}) gives rise to a unique facet.
\end{proof}

\begin{figure}[htbp]
\begin{center}
\includegraphics[scale=.6]{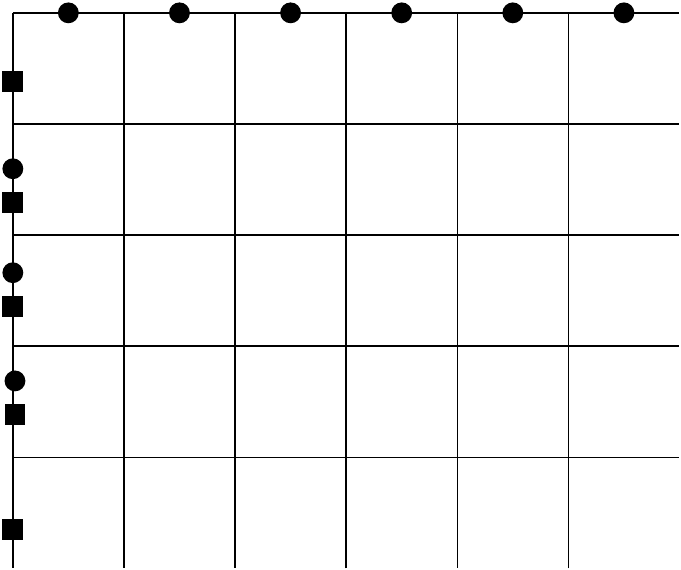}
\end{center}
\caption{$\Gamma_{(m,n)}$ decorated with symbols that represent the inequalities that do not determine facets of $P(m,n)$. Squares represent partial column sums of the form $\sum X_{ij}\leq 1$ and dots represent partial row or column sums of the form $\sum X_{ij}\geq 0$.}
\label{fig:facet_mn}
\end{figure}

We now state a theorem on the number of facets of $P(\lambda,n)$. We then give simpler formulas as corollaries in the special cases of two-row shapes, rectangles, and hooks. First, recall that $k$ is the number of parts of $\lambda$, and that $a_{\lambda_1}$ is the number of parts in $\lambda$ of size $\lambda_1$.

\begin{theorem}
\label{thm:facets_lambda}
The number of facets of $P(\lambda,n)$ is: 
\begin{center}
\begin{equation}
\label{eq:facets}
3n\lambda_1-n-3(\lambda_1-1)-(n-2)(\lambda_1-\lambda_2+\lambda_{n-1})
-(k-a_{\lambda_1})
-2(\lambda_1-D(\lambda)) - C(\lambda)
\end{equation}
   \end{center}
where $D(\lambda)$ is the number of distinct part sizes of $\lambda$ (each part size counts once, even though there may be multiple parts of a given size), we take $\lambda_i=0$ if $k<i$, and $C(\lambda)$ equals the following:
\[C(\lambda)=\begin{cases} 
      2 & \mbox{ if } \ k=1, \\
      1 & \mbox{ if } \ 1<k<n-1 \ \mbox{and} \ \lambda_1\neq\lambda_2, \\
      0 & \mbox{ if } \ 1<k<n-1 \ \mbox{and} \ \lambda_1=\lambda_2, \\
      2 & \mbox{ if } \ k=n-1 \ \mbox{and either} \ \lambda_1\neq\lambda_2 \ \mbox{or} \ \lambda=\lambda_1^{k}, \\
      1 & \mbox{ if } \ k=n-1, \ \lambda_1=\lambda_2, \ \mbox{and} \ \lambda\neq \lambda_1^{k}. \\
   \end{cases}\]
\end{theorem}
\begin{proof}
By Theorem~\ref{thm:facets_mn}, since $P(\lambda,n)$ satisfies all the inequalities satisfied by $P(m,n)$ for $m=\lambda_1$, we have at most $3n\lambda_1 - n -2(\lambda_1-1)$ facets, given by the equalities (\ref{eq:f1})--(\ref{eq:f5}). See Figure~\ref{fig:facet_mn}.

But note equalities of the form (\ref{eq:f1}) with $j=n$ no longer give facets, since by (\ref{eq:eq3}) the total sum of each matrix row is fixed. There are $\lambda_1-1$ such inequalities, so we now have at most $3\lambda_1 n - n -3(\lambda_1-1)$ facets. See Figure~\ref{fig:facet_shape1}.

To prove our count in (\ref{eq:facets}), we determine which of the remaining equalities in (\ref{eq:f1})--(\ref{eq:f5}) are unnecessary. We discuss each remaining term of (\ref{eq:facets}) below. Let $X\in P(\lambda,n)$. 

\begin{enumerate}
\item
\label{bullet1}$-{(n-2)(\lambda_1-\lambda_2+\lambda_{n-1})}$:
First, suppose $\lambda_1\neq \lambda_2$, otherwise $(n-2)(\lambda_1-\lambda_2)=0$. Since $\lambda_1\neq\lambda_2$, the first row of $X$ sums to 1 and the next $\lambda_1-\lambda_2-1$ rows sum to 0. So the first $i$ rows all together sum to 1 for any $1\leq i\leq\lambda_1-\lambda_2$. That is, for any fixed $i\in[1,\lambda_1-\lambda_2]$, $\displaystyle\sum_{i'=1}^{i}\displaystyle\sum_{j'=1}^n X_{i'j'}=1$. Also, by (\ref{eq:eq1}), $\displaystyle\sum_{i'=1}^{i} X_{i'j}\geq 0$, and by (\ref{eq:eq2}), $\displaystyle\sum_{j'=1}^{j} X_{ij'}\geq 0$.  So we have the following sum:
\[1=\displaystyle\sum_{i'=1}^{i}\displaystyle\sum_{j'=1}^n X_{i'j'}=\displaystyle\sum_{i'=1}^{i}\underbrace{\displaystyle\sum_{j'=1}^{j-1} X_{i'j'}}_{\geq 0}+\displaystyle\sum_{j'=j}^{n} \underbrace{\displaystyle\sum_{i'=1}^{i} X_{i'j'}}_{\geq 0}.
\]
Since we have all positive terms summing to 1, none of these terms may exceed 1. Therefore, $\displaystyle\sum_{i'=1}^{i} X_{i'j}\leq 1$ for all $1\leq i\leq\lambda_1-\lambda_2$, $1\leq j\leq n$. 

Thus the partial sums of the form $\displaystyle\sum_{i'=1}^i X_{i'j}\leq 1$ for $1\leq i\leq\lambda_1-\lambda_2$, $1\leq j\leq n$ are unnecessary. We have already disregarded these inequalities for $j=1$, $1\leq i\leq n-1$ in Theorem~\ref{thm:facets_mn}. We will consider $j=1$, $i=n$ in (\ref{bullet4}). We will count the partial column sums in the $n$th column in (\ref{bullet3}). Thus, for this term we count the $(n-2)(\lambda_1-\lambda_2)$ unnecessary inequalities $\displaystyle\sum_{i'=1}^i X_{i'j}\leq 1$ for $1\leq i\leq\lambda_1-\lambda_2$, $2\leq j\leq n-1$. 

Now suppose $k=n-1$ so that $\lambda_{n-1}\neq 0$, otherwise $(n-2)\lambda_{n-1}=0$. 
Since $\lambda_{n-1}\neq 0$, the last $\lambda_{n-1}$ rows of $X$ sum to $0$. That is, for any fixed $i\in[\lambda_1-\lambda_{n-1}+1,\lambda_1]$, $\displaystyle\sum_{i'=i+1}^{\lambda_1}\displaystyle\sum_{j'=1}^n X_{i'j'}=0$. Also, by (\ref{eq:eq3}), $\displaystyle\sum_{i'=1}^{\lambda_1}\displaystyle\sum_{j'=1}^n X_{i'j'}=\displaystyle\sum_{i'=1}^{\lambda_1}a_{\lambda_1-i+1}=k=n-1$, since $\lambda$ has $n-1$ parts.
Also, by (\ref{eq:eq1}), $\displaystyle\sum_{i'=1}^{i} X_{i'j}\leq 0$.
So we have the following sum:
\[n-1=\displaystyle\sum_{j'=1}^n\displaystyle\sum_{i'=1}^{\lambda_1}X_{i'j'}=\displaystyle\sum_{j'=1}^n\underbrace{\displaystyle\sum_{i'=1}^{i} X_{i'j'}}_{\leq 1}+\underbrace{\displaystyle\sum_{j=1}^{n} \displaystyle\sum_{i'=i+1}^{\lambda_1} X_{i'j'}}_{=0}
\]
Since we have $n$ terms $\displaystyle\sum_{i'=1}^{i} X_{i'j'}$  summing to $n-1$, each at most $1$, none of these terms may be negative. Therefore, $\displaystyle\sum_{i'=1}^{i} X_{i'j}\geq 0$ for all $\lambda_1-\lambda_{n-1}+1\leq i\leq\lambda_1$, $1\leq j\leq n$. 

Thus the partial sums of the form $\displaystyle\sum_{i'=1}^i X_{i'j}\geq 0$ for $\lambda_1-\lambda_{n-1}+1\leq i\leq\lambda_1$, $1\leq j\leq n$ are unnecessary. We have already disregarded these inequalities for $j=1$, $2\leq i\leq n$ in Theorem~\ref{thm:facets_mn}. 
We will count the partial column sums in the $n$th column in (\ref{bullet3}). Thus, for this term we count the $(n-2)\lambda_{n-1}$ unnecessary inequalities $\displaystyle\sum_{i'=1}^i X_{i'j}\leq 1$ for $\lambda_1-\lambda_{n-1}+1\leq i\leq\lambda_1$, $2\leq j\leq n-1$. See Figure~\ref{fig:facet_shape1}.

\smallskip
\item \label{bullet2}$-(k-a_{\lambda_1})$:
Let $i>1$.
By (\ref{eq:eq3}), $\displaystyle\sum_{j'=1}^{n} X_{i j'}= a_{\lambda_1-i+1}$. Now $0\leq \displaystyle\sum_{i'=1}^{i-1} X_{in}$ and $\displaystyle\sum_{i'=1}^{i} X_{in}\leq 1$ imply $X_{in}\leq 1$, so we have $\displaystyle\sum_{j'=1}^{n-1} X_{i j'}\geq a_{\lambda_1-i+1}-1$. 
This implies the inequality $\displaystyle\sum_{j'=1}^{n-1} X_{i j'}\geq 0$ whenever $a_{\lambda_1-i+1}>0$. Similarly, $\displaystyle\sum_{j'=1}^{n-z} X_{i j'}\geq a_{\lambda_1-i+1}-z$ for all $1\leq z\leq a_{\lambda_1-i+1}$ since the last $z$ entries in that row sum to at most $z$ (since entries can be no more than 1, by the column partial sums). Thus, the $a_{\lambda_1-i+1}$ inequalities $\displaystyle\sum_{j'=1}^{n-z} X_{i j'}\geq 0$, $1\leq z\leq a_{\lambda_1-i+1}$, are unnecessary. By reindexing, this is equivalent to $\displaystyle\sum_{j'=1}^{j} X_{i j'}\geq 0$, $n-a_{\lambda_1-i+1}\leq j\leq n-1$.

We already discarded all the row partial sum inequalities in the first row in Theorem~\ref{thm:facets_mn}, so we do not count those here. Thus $a_{\lambda_1}$ is not included.
So we have ${\displaystyle\sum_{i'=1}^{\lambda_1-1} a_{i'}}$ unnecessary partial sum inequalities. This equals the total number of parts of $\lambda$ minus the number of parts with part size $\lambda_1$, that is, $k-a_{\lambda_1}$. See Figure~\ref{fig:facet_shape3}.

\smallskip
\item \label{bullet3} $-2(\lambda_1-D(\lambda))$:
Suppose $a_{\lambda_1-i+1}=0$ so that the total sum of row $i$ of $X$ equals 0. Then the last entry $X_{in}$ may not be greater than $0$, since this would contradict $\displaystyle\sum_{j'=1}^{n-1} X_{ij'}\geq 0$. So the inequality $\displaystyle\sum_{i'=1}^i X_{i'n} \leq 1$ is unnecessary.
Also, since the total sum of row $i$ of $X$ equals 0, we have
then $X_{in}= -\displaystyle\sum_{j=1}^{n-1} X_{ij}$. In addition, $\displaystyle\sum_{i'=1}^i X_{i'n}\geq 0$. We substitute the previous equality into this inequality to obtain $\displaystyle\sum_{i'=1}^{i-1} X_{i'n}-\displaystyle\sum_{j=1}^{n-1} X_{ij}\geq 0$. We know $\displaystyle\sum_{j=1}^{n-1} X_{ij}\geq 0$, so this implies $\displaystyle\sum_{i'=1}^{i-1} X_{i'n}\geq 0$. 

So for each $a_{\lambda_1-i+1}=0$ we have two unnecessary inequalities: $\displaystyle\sum_{i'=1}^i X_{i'n} \leq 1$ and $\displaystyle\sum_{i'=1}^{i-1} X_{i'n}\geq 0$. The number of row sums equal to zero is given by the number of integers $\ell$ with $1\leq\ell\leq\lambda_1$ such that $a_{\ell}=0$. This count equals $\lambda_1-D(\lambda)$, where $D(\lambda)$ equals the number of distinct part sizes of $\lambda$. Thus, we have $2(\lambda_1-D(\lambda))$ unnecessary inequalities. See Figure~\ref{fig:facet_shape3}.

\smallskip
\item \label{bullet4} $-C(\lambda)$: We now have a few more border inequalities to discard, depending on $\lambda$. We take each case in turn. See Figure~\ref{fig:facet_shape4}.

\begin{enumerate}
\item \label{a}
When $\lambda_1\neq\lambda_2$, we may also discard the inequality $X_{1 n} \leq 1$, as this is a partial sum of the form $\displaystyle\sum_{i'=1}^i X_{i'n} \leq 1$ for $1\leq i\leq\lambda_1-\lambda_2$, which by reasoning in (\ref{bullet1}) may be discarded. The other inequalities of that form have already been counted in (3), thus we have one additional unnecessary inequality whenever $\lambda_1\neq\lambda_2$. Note, since $\lambda_2=0\neq\lambda_1$ for $k=1$, this inequality is also discarded in the case $k=1$.

\item \label{b}
When $k=1$, 
since $\displaystyle\sum_{j'=1}^n X_{1j'}=1$ and $\displaystyle\sum_{j'=1}^n X_{ij'}=0$ for all $2\leq i\leq\lambda_1$, we have that the sum of all the entries in the matrix is $1$. This, together with the inequalities $\displaystyle\sum_{i'=1}^{\lambda_1} X_{i'j}\geq 0$, $2\leq j\leq n$, implies $\displaystyle\sum_{i'=1}^{\lambda_1} X_{i'1}\leq 1$. 
So we have one additional unnecessary inequality when $k=1$. 

\item 
\label{c} When $1<k=n-1$, 
by the reasoning in the $k=n-1$ case of (\ref{bullet1}) we may discard the inequality $\displaystyle\sum_{i'=1}^{\lambda_1} X_{i'n}\geq 0$. If $k=1$, $n=2$, we may not discard this inequality, since in this case we have already discarded the inequality in (\ref{b}).

\item \label{d} Suppose $k=n-1$ and $\lambda$ is a rectangle, so $\lambda_{n-1}=\lambda_1$. In this case, we may also discard the inequality $X_{1 1} \geq 0$; this is  a partial sum of the form $\displaystyle\sum_{i'=1}^i X_{i'1} \geq 0$ for $\lambda_1-\lambda_{n-1}+1\leq i\leq\lambda_1$ which by the reasoning in (\ref{bullet1}) may be discarded. The other inequalities of that form have already been counted in (3), thus we have one additional unnecessary inequality whenever $\lambda_1=\lambda_1^{n-1}$ and $k>1$. If $k=1$, $n=2$, we may not discard this inequality, since we have already discarded the inequality in (\ref{a}).
\end{enumerate}
\end{enumerate}

\smallskip
Thus the total number of facets is at most (\ref{eq:facets}).
We claim this upper bound is the facet count. That is, a facet can be defined as all ${X} \in P(\lambda,n)$ which satisfy exactly one of the following: 
\begin{align}
\label{eq:z1}
r_{ij} &= \displaystyle\sum_{j'=1}^{j} X_{ij'} = 0, \hspace{.3in} & 2 \leq i \leq \lambda_1 \text{ and }
1 \leq j \leq n-a_{\lambda_1-i+1}-1 \\
\label{eq:z2}
c_{ij} &= \displaystyle\sum_{i'=1}^{i} X_{i'j} =0, \hspace{.3in} & 1 \leq i \leq \lambda_1 \text{ and } 2 \leq j \leq n-1 
\\
\label{eq:z2a}
c_{in } &= \displaystyle\sum_{i'=1}^{i} X_{i'n} =0, \hspace{.3in} & (i=\lambda_1 \mbox{ and } k<n-1) \mbox{ or } (1 \leq i \leq \lambda_1-1 \text{ and } a_{\lambda_1-i}>0)
\\
\label{eq:z3a}
c_{ij} &=\displaystyle\sum_{i'=1}^{i} X_{i'j}=1, \hspace{.3in} & \lambda_1-\lambda_2+1 \leq i \leq \lambda_1 \text{ and } 2 \leq j \leq n-1 \\
\label{eq:z3n}
c_{in} &=\displaystyle\sum_{i'=1}^{i} X_{i'n}=1, \hspace{.3in} & \lambda_1-\lambda_2+1 \leq i \leq \lambda_1  \text{ and } a_{\lambda_1-i+1}>0\\
\label{eq:z4}
r_{11} &= c_{11} = X_{11} =0 & \text{if } \lambda=\lambda_1^{n-1} \text{ and } k>1\\
\label{eq:z5}
c_{\lambda_1 1} &= \displaystyle\sum_{i'=1}^{\lambda_1} X_{i'1}=1 & \text{if } k=1.
\end{align}

Note each equality fixes exactly one matrix entry, lowering the dimension by one. By an argument similar to that given in Theorem~\ref{thm:facets_mn}, given any two equalities above, we may construct a sign matrix in $M(\lambda,n)$ that satisfies one but not the other.
\end{proof}

\begin{figure}[htbp]
\begin{center}
\includegraphics[scale=.6]{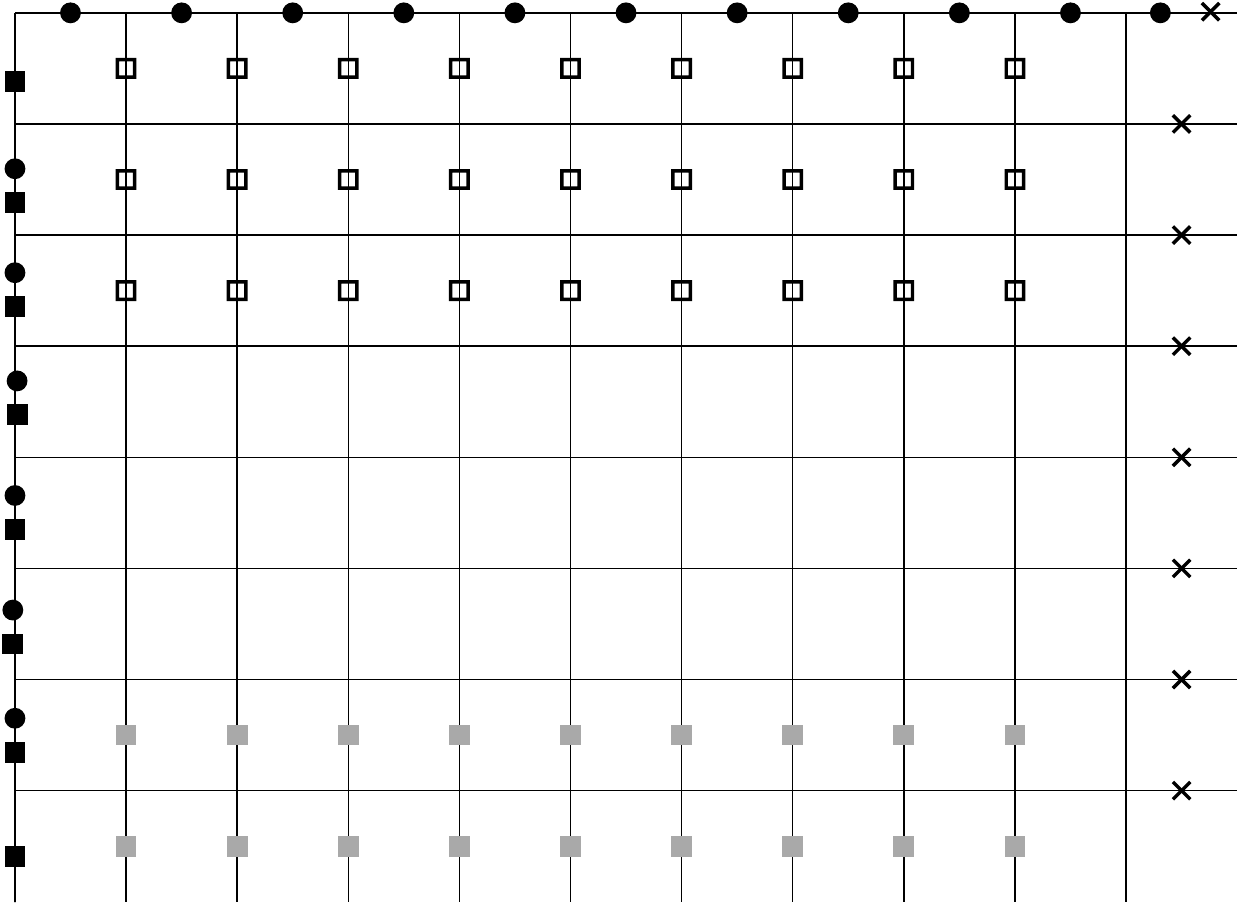}
\end{center}
\caption{$\Gamma_{(\lambda_1,n)}$ decorated with symbols that represent inequalities that do not determine facets of $P(\lambda,n)$. Squares represent partial column sums of the form $\sum X_{ij}\leq 1$ and dots represent partial row or column sums of the form $\sum X_{ij}\geq 0$.
The filled-in shapes represent inequalities that were already removed in the facet proof for $P(m,n)$. The crosses represent the fixed row sums in $P(\lambda, n)$. The open squares and gray squares represent inequalities that are removed in (\ref{bullet1}) from the proof of Theorem~\ref{thm:facets_lambda}.}
\label{fig:facet_shape1}
\end{figure}

\begin{figure}[htbp]
\begin{center}
\begin{multicols}{2}
\includegraphics[scale=.6]{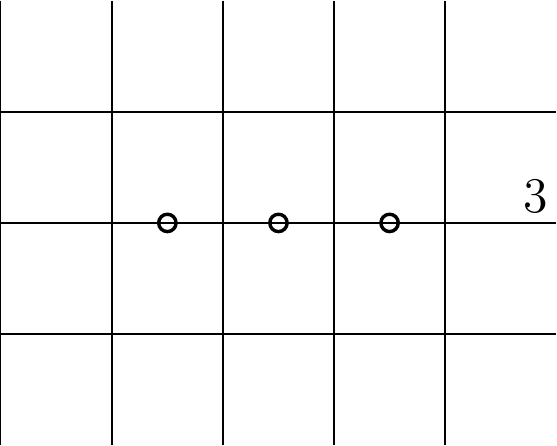}

\includegraphics[scale=.6]{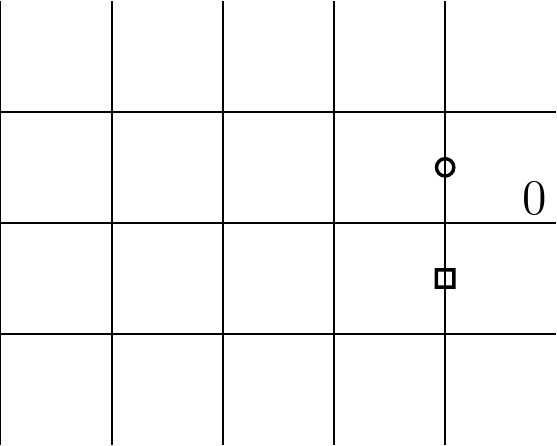}
\end{multicols}
\end{center}
\caption{Examples of portions of 
$\Gamma_{(\lambda_1,n)}$ 
that represent inequalities removed based on the fixed row sums. The left diagram shows removed inequalities discussed in (\ref{bullet2}) in the case 
$a_{\lambda_{1-i+1}}>0$. 
The right diagram shows removed inequalities discussed in (\ref{bullet3}) in the case $a_{\lambda_{1-i+1}}=0$.
}
\label{fig:facet_shape3}
\end{figure}

\begin{figure}[htbp]
\begin{center}
\includegraphics[scale=.6]{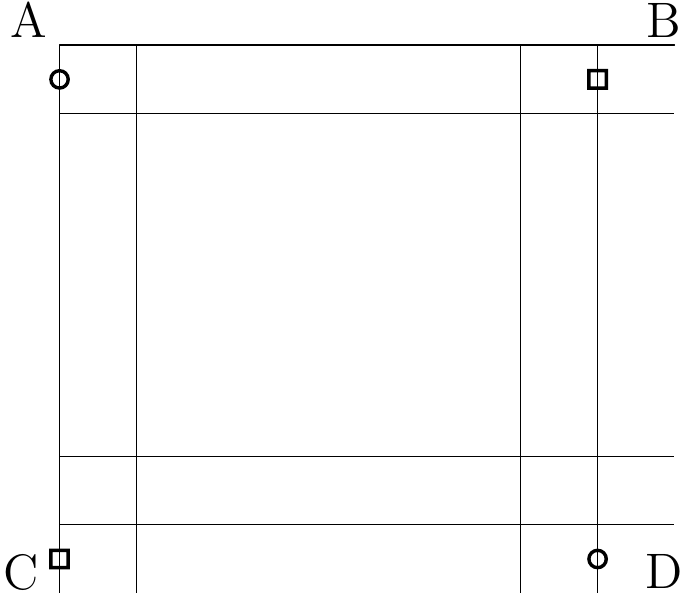}
\end{center}
\caption{$\Gamma_{(\lambda_1,n)}$ decorated with symbols that represent inequalities  removed in (\ref{bullet4}) from the proof of Theorem~\ref{thm:facets_lambda}. 
A is discussed in (\ref{d}),  B is discussed in (\ref{a}),  C is discussed in (\ref{b}), and  D is discussed in (\ref{c}). }
\label{fig:facet_shape4}
\end{figure}

\begin{corollary}
The number of facets of $P([\lambda_1,\lambda_2],n)$ when $\lambda_1 \ne \lambda_2$ is as follows:
\begin{itemize}
\item 
$3n\lambda_1-n-5(\lambda_1-1)-(n-2)(\lambda_1-\lambda_2)$, when $n>3$;
\item $3n\lambda_1-n-5(\lambda_1-1)-(n-2)\lambda_2-1$, when $n=3$.
\end{itemize}
\end{corollary}
\begin{proof}
Suppose $\lambda_1\neq\lambda_2$ and $n > 3$. Then $a_{\lambda_1}=1$, $D(\lambda)=2$, and $C(\lambda)=1$ (from the definition in Theorem~\ref{thm:facets_lambda}). Thus since $k=2$, the formula of Theorem~\ref{thm:facets_lambda} specializes to
$3n\lambda_1-n-3(\lambda_1-1)-(n-2)(\lambda_1-\lambda_2)-(2-1)-2(\lambda_1-2)-1$, which reduces to the above formula.
Now suppose $\lambda_1\neq\lambda_2$ and $n = 3$. In this case $\lambda_{n-1}=\lambda_2$ and $C(\lambda)=2$ but the rest of the values remain the same. Thus, the formula of Theorem~\ref{thm:facets_lambda} specializes to the above.
\end{proof}

In the above corollary, we required $\lambda_1 \ne \lambda_2$. The case $\lambda=[\lambda_1, \lambda_1]=[\lambda_1^2]$ is a special case of the next corollary, which enumerates the facets when $\lambda$ is a rectangle. 

\begin{corollary}
\label{cor:squares}
The number of facets of $P(\lambda_1^k,n)$ is as follows: 
\begin{itemize}
\item 0, when $k=n$;
\item $2n\lambda_1-n-3(\lambda_1-1)$, when $k=n-1$ or $k=1$;
\item $3n\lambda_1-n-5(\lambda_1-1)$, when $1<k<n-1$.
\end{itemize}
\end{corollary}

\begin{proof}
Suppose $k=n$.
By Proposition~\ref{prop:dim}, since $k=n$ we have that the dimension of $P(\lambda_1^k)$ equals $(\lambda_1-\lambda_n)(n-1)=(\lambda_1-\lambda_1)(n-1)=0$. Since the polytope is zero dimensional, there are no facets.

Suppose $k=n-1$. We then have the following: $\lambda_1=\lambda_2= \lambda_{n-1}$, $a_{\lambda_1}=k=n-1$, $D(\lambda)=1$, and $C(\lambda)=2$. Therefore by Theorem~\ref{thm:facets_lambda} the number of facets is $3n\lambda_1-n-3(\lambda_1-1)-(n-2)(0+\lambda_1)-0-2(\lambda_1-1)-2$ which reduces to the formula above.

For $1<k<n-1$, by Theorem~\ref{thm:facets_lambda} the number of facets is
$3n\lambda_1-n-3(\lambda_1-1)-(n-2)(\lambda_1-\lambda_2+\lambda_{n-1})
-(k-a_{\lambda_1})-2(\lambda_1-D(\lambda)) - C(\lambda)$. Since $\lambda_1=\lambda_2$ and $\lambda_{n-1}=0$, the 4th term equals $0$. The 5th term equals $0$ since $a_{\lambda_1}=k$. Note $D(\lambda_1^k)=1$, so the 6th term equals $2(\lambda_1-1)$. $C(\lambda)=0$, so the resulting count follows.

When $k=1$,  by Theorem~\ref{thm:facets_lambda} the number of facets is
$3n\lambda_1-n-3(\lambda_1-1)-(n-2)(\lambda_1-\lambda_2+\lambda_{n-1})
-(k-a_{\lambda_1})-2(\lambda_1-D(\lambda)) - C(\lambda) = 3n\lambda_1-n-3(\lambda_1-1)-(n-2)(\lambda_1-0+0)-(1-1)-2(\lambda_1-1) - 2$, since $a_{\lambda_1}=D(\lambda)=1$ and $\lambda_2=0$. The resulting count follows.
\end{proof}

Finally, we have the following corollary in the case that $\lambda$ is hook-shaped.
\begin{corollary} 
The number of facets of $P([\lambda_1,1^{k-1}],n)$ is as follows:
\begin{itemize}
\item 
$2n(\lambda_1-1)-n-3(\lambda_1-2)$,
when $k=n$;
\item $2n\lambda_1-2n-3(\lambda_1-1)+4$, 
when $k=n-1$;
\item $2n\lambda_1-3(\lambda_1-1)-k+2$, when $1<k <n-1$.
\end{itemize}
\end{corollary}

\begin{proof}
When $k=n$, the first column of the tableau corresponding to any sign matrix in the polytope is fixed as $1, 2, \ldots, n$, so this reduces to the case of rectangles of one row, that is, shape $[\lambda_1-1]$. So by Corollary~\ref{cor:squares}, we have $2n(\lambda_1-1)-n-3((\lambda_1-1)-1)=2n\lambda_1-3n-3\lambda_1+6$ facets.

When $k=n-1$, in the formula in Theorem~\ref{thm:facets_lambda} we have that $\lambda_2=\lambda_{n-1}=1$, $a_{\lambda_1}=1$, $D(\lambda)=2$ and $C(\lambda)=2$. Therefore, by Theorem~\ref{thm:facets_lambda} the number of facets is $3n\lambda_1-n-3(\lambda_1-1)-(n-2)(\lambda_1-1+1)-(n-1-1)-2(\lambda_1-2)-2$, which when simplified yields the desired result.

When $1<k<n-1$, $a_{\lambda_1}=1$, $D(\lambda)=2$, and $\lambda_1\neq\lambda_2$ so $C(\lambda)=1$.
So by Theorem~\ref{thm:facets_lambda} the number of facets is
$3n\lambda_1-n-3(\lambda_1-1)-(n-2)(\lambda_1-1)-(k-1)-2(\lambda_{1}-2)
-1$, which when simplified yields the desired result. 
\end{proof}

\section{Face lattice descriptions}
\label{sec:facelattice}

In this section, we determine the face lattice of the $P(m,n)$ and $P(\lambda,n)$ polytope families. We also show that given any two faces, 
we may determine the smallest dimensional face in which they are contained. The ideas for proving the face lattice were inspired by \cite{striker} and \cite{heean}.

\begin{definition}[\cite{ziegler}]
The \emph{face lattice} of a convex polytope $P$ is the poset $L:= L(P)$ of all faces of $P$,  partially ordered by inclusion. 
\end{definition}

\begin{definition}
We define the \emph{complete partial sum graph} denoted $\overline\Gamma_{(m,n)}$ as the following labeling of the graph $\Gamma_{(m,n)}$. The horizontal edges are labeled with $\{0,\star\}$, 
while the vertical edges are labeled $\{0,1,\{0,1\}\}$. An example is shown for $P(3,5)$ in Figure~\ref{fig:completegraph}.
\label{def:completegraph}
\end{definition}

\begin{figure}[htbp]
\begin{center}
\includegraphics[scale=.7]{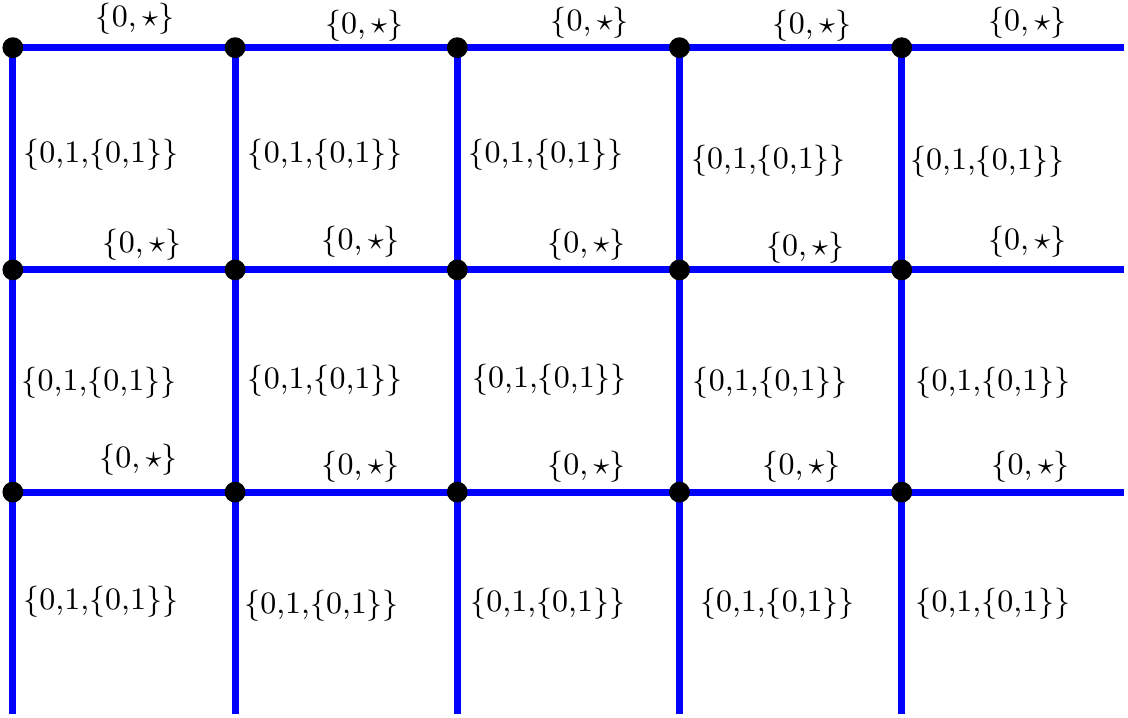}
\end{center}
\caption{The complete partial sum graph  $\overline\Gamma_{(3,5)}$.}
\label{fig:completegraph}
\end{figure}

\begin{definition}
\label{def:0-dim}
A \emph{0-dimensional} component of $\overline\Gamma_{(m,n)}$ is a labeling of $\Gamma_{(m,n)}$ such that the edge labels are one element subsets of the edge labels of $\overline\Gamma_{(m,n)}$ and such that the edge labels come from the partial sums of a sign matrix as follows: Let the edges be labeled as in $\thickhat{M}$ for some $m\times n$ sign matrix $M$, with the exception that horizontal edges labeled by nonzero numbers in $\thickhat{M}$ are now labeled as $\star$. 
For any $m\times n$ sign matrix $M$, let $g(M)$ be the $0$-dimensional component associated to $M$.
\end{definition}

\begin{lemma}
\emph{0-dimensional} components of $\overline\Gamma_{(m,n)}$ are in bijection with $m\times n$ sign matrices.
\end{lemma}
\begin{proof}
Recall we may recover a sign matrix $M$ from its column partial sums. Thus, even though we are not keeping the exact values of the row partial sums, we still have enough information to recover a sign matrix $M$ from $g(M)$. Thus,
given sign matrices $M_1\neq M_2$, $g(M_1)\neq g(M_2)$.
\end{proof}

\begin{definition}
\label{def:union}
Let $\delta$ and $\delta'$ be labelings of  $\Gamma_{(m,n)}$ such that the edge labels are subsets of the corresponding edge label sets in $\overline\Gamma_{(m,n)}$.
Define the \emph{union} $\delta\cup \delta'$  as the labeling of $\Gamma_{(m,n)}$ such that each edge is labeled by the union of the corresponding labels on $\delta$ and $\delta'$, where we consider $0\cup\star=\star$. Define the \emph{intersection} $\delta\cap\delta'$  to be a labeling of $\Gamma_{(m,n)}$ such that each edge is labeled by the intersection of the corresponding labels on $\delta$ and $\delta'$, where we consider $0\cap\star=0$. So the vertical edges will have labels of $\emptyset, 0, 1$, or $\{0,1\}$ and the horizontal edges will have labels of $0$ or $\star$. In our figures, vertical edges labeled $\{0,1\}$ and horizontal edges labeled $\star$ will be darkened (blue).
\end{definition}

\begin{definition}
Let $\delta$ be a labeling of $\Gamma_{(m,n)}$ such that the edge labels are subsets of the corresponding edge label sets in $\overline\Gamma_{(m,n)}$.
\begin{enumerate}
\item $\delta$ is a \emph{component} of $\overline\Gamma_{(m,n)}$ if it is either the empty labeling of $\Gamma_{(m,n)}$ (we call this the \emph{empty component} denoted $\emptyset$) or if it can be presented as the union of any set of $0$-dimensional components.
\item For two components $\delta$ and $\delta'$ of $\overline\Gamma_{(m,n)}$, we say $\delta$ is a \emph{component of} $\delta'$ if the edge labels  of $\delta$ are each a subset of the corresponding edge labels  
of $\delta'$, where we consider $0$ to be a subset of $\star$.
\end{enumerate}
\label{def:face}
\end{definition}

\begin{remark}
\label{remark:union}
Note if $\delta$ and $\delta'$ are components of $\overline{\Gamma}_{(m,n)}$, $\delta\cup\delta'$ is also a component. This is because each of $\delta$ and $\delta'$ is a union of 0-dimensional components, so $\delta\cup\delta'$ is as well.
\end{remark}

\begin{figure}
\begin{center}

\includegraphics[scale=.7]{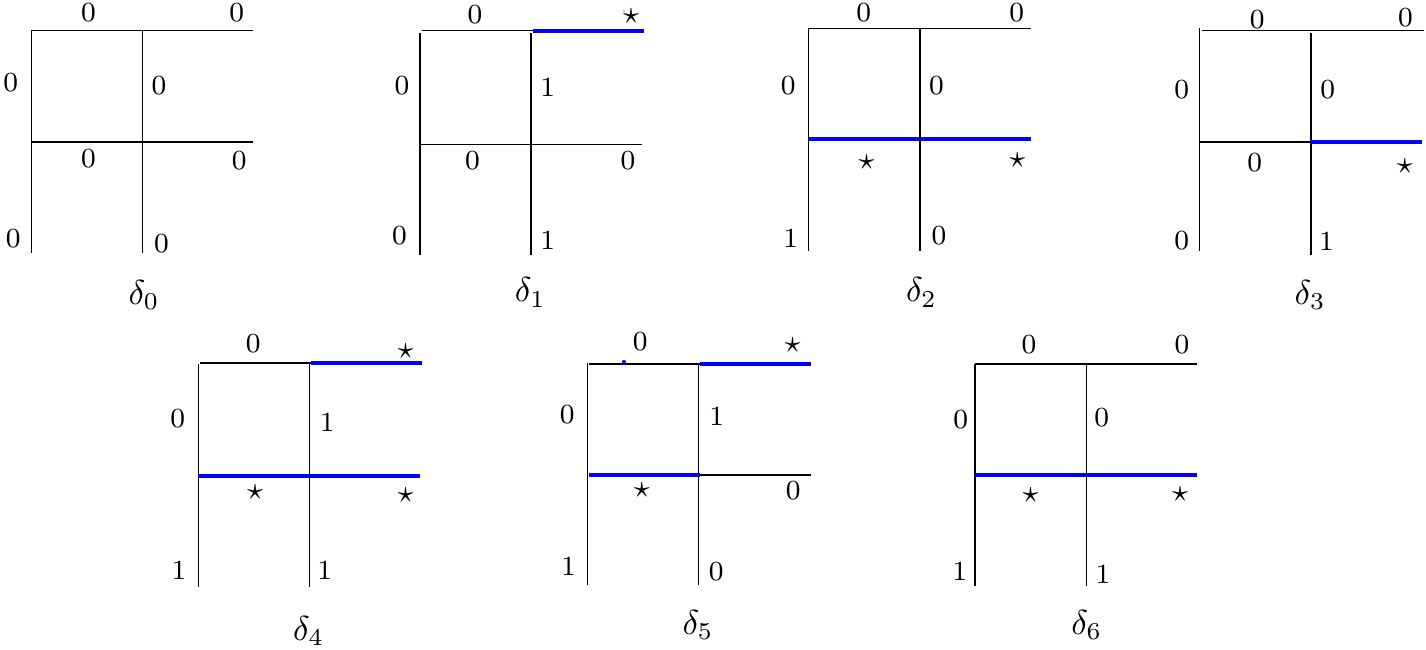}

Seven of the $\overline\Gamma_{(2,2)}$ $0$-dimensional components

\includegraphics[scale=.75]{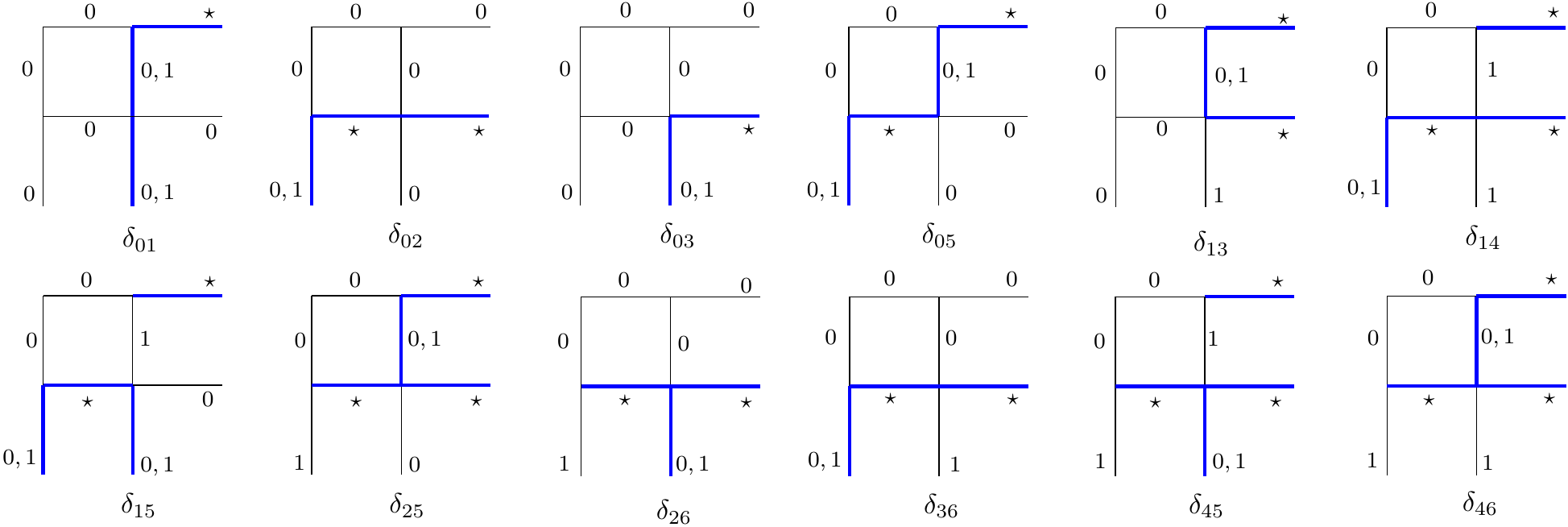}

Twelve of the $\overline\Gamma_{(2,2)}$ $1$-dimensional components

\includegraphics[scale=.75]{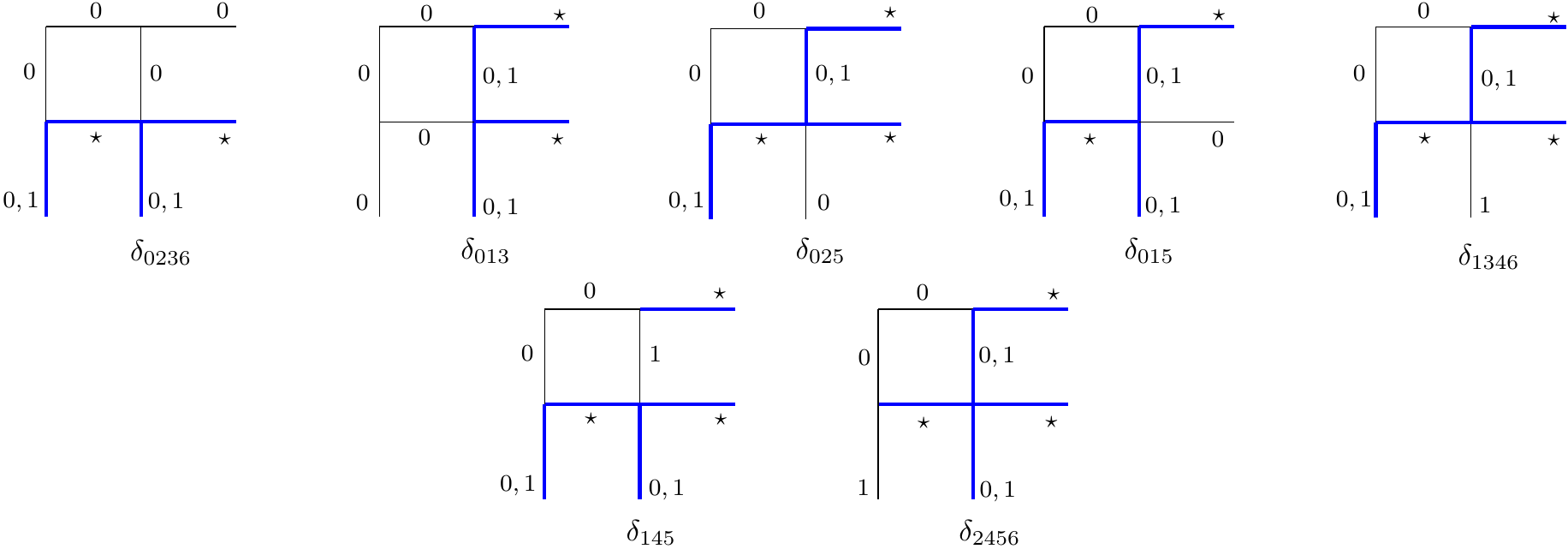}

Seven of the $\overline\Gamma_{(2,2)}$ $2$-dimensional components

\includegraphics[scale=.8]{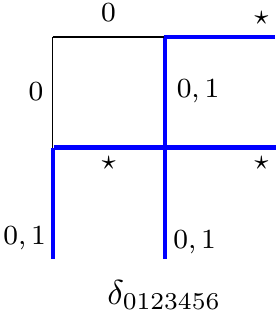} \hspace{1.3in} \includegraphics[scale=.8]{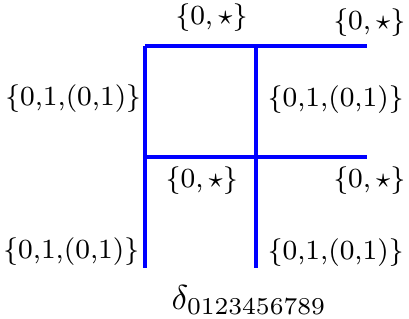}

One of the $3$-dimensional components \hspace{.7in} The complete partial sum graph $\overline\Gamma_{(2,2)}$

\end{center}
\caption{A set of components of $\overline\Gamma_{(2,2)}$.}
\label{fig:faces}
\end{figure}

Next, we define a partial order on components of $\overline\Gamma_{(m,n)}$.

\begin{definition}
Define a {partial order} $\Lambda_{(m,n)}$ on components of $\overline\Gamma_{(m,n)}$  by containment. That is, $\delta\leq \delta'$ in $\Lambda_{(m,n)}$ if and only if $\delta$ is a component of $\delta'$. Say $\delta'$ \emph{covers} $\delta$, denoted $\delta\lessdot\delta'$, if $\delta$ is contained in $\delta'$ and there is no component $\delta''$ of $\overline\Gamma_{(m,n)}$ such that $\delta<\delta''<\delta'$.
\end{definition}

\begin{remark}
\label{remark:meet_join}
For components $\delta$ and $\delta'$ of $\overline\Gamma_{(m,n)}$, we may 
define  $\delta \vee \delta' = \delta\cup\delta'$. By Remark~\ref{remark:union}, this is itself a component of $\overline\Gamma_{(m,n)}$. Also, it is the smallest component containing both $\delta$ and $\delta'$ as subcomponents, so this is the \emph{join operator} of $\Lambda_{(m,n)}$. 
We will show in Theorems~\ref{th:g_bijection} and \ref{thm:poset_iso} that 
$\Lambda_{(m,n)}$  is the face lattice of $P(m,n)$, thus there also exists a well-defined meet operator, since $\Lambda_{(m,n)}$ is a lattice. 
The {meet} $\delta \wedge \delta'$  will be the maximal component contained in the intersection $\delta \cap \delta'$; note this could be the empty component.
\end{remark}

\begin{remark}
\label{remark:max}
Note the maximal component of $\Lambda_{(m,n)}$ is the union of all $0$-dimensional components. Thus, it has labels $\{0,1\}$ on the vertical edges of $\Gamma_{(m,n)}$ and $\star$ on the horizontal edges.
\end{remark}

\begin{example}
We show examples of several of the above definitions using Figure~\ref{fig:faces} (which by the upcoming Theorems~\ref{th:g_bijection} and \ref{thm:poset_iso} is the face lattice of one of the $3$-dimensional faces of $P(2,2)$). 

\begin{enumerate}
\item[i).] We first exhibit a component as a union of $0$-dimensional components: $\delta_{025}=\delta_0 \cup \delta_2 \cup \delta_5$. 
\item[ii).] We now show how the union of two components can contain more $0$-dimensional components than are contained in the original component: $\delta_{14} \cup \delta_{46} = \delta_{0123456}$. Note $\delta_{0123456}$ is the join.
\item[iii).] 
Next we intersect two components: $\delta_{2456} \cap \delta_{015}= \delta_5$. Note $\delta_5$ is the meet.
\item[iv).] To illustrate containment of components, note the $1$-dimensional components $\delta_{01}, \delta_{03}$, and $\delta_{13}$ are all contained in the $2$-dimensional component $\delta_{013}$.
\end{enumerate}
\end{example}

\begin{definition}
\label{def:region}
Given a component  $\delta\in\Lambda_{(m,n)}$, consider the planar graph $G$ composed of the darkened edges of $\delta$; we regard any darkened edges on the right and bottom as meeting at a point in the exterior region.
We say a \emph{region} of $\delta$ is defined as a planar region of $G$, excluding the exterior region. 
Let $\mathcal{R}(\delta)$ denote the number of regions of $\delta$.  
For consistency we set $\mathcal{R}(\emptyset)=-1$.
\end{definition}


See Figure~\ref{fig:region} for an example of this definition.

\smallskip
We now state a lemma which shows that moving up in the partial order $\Lambda_{(m,n)}$ increases the number of regions. We will use this lemma in the proof of Theorem~\ref{thm:poset_iso}.

\begin{lemma}
\label{prop:regions}
Suppose a component $\delta\in\Lambda_{(m,n)}$ has $\mathcal{R}(\delta)=\omega$. If $\delta\lessdot\delta'$ then $\mathcal{R}(\delta')\geq \omega+1$. 
\end{lemma}

\begin{proof}
By convention, the empty component has $\mathcal{R}(\emptyset)=-1$. If $\delta$ is a $0$-dimensional component, $\mathcal{R}(\delta)=0$, as there are no regions in a $0$-dimensional component. Suppose a component $\delta\in\Lambda_{(m,n)}$ has $\mathcal{R}(\delta)=\omega$. We wish to show if $\delta\lessdot\delta'$ then $\mathcal{R}(\delta')\geq \omega+1$.  $\delta\lessdot\delta'$ implies that the labels of each edge of $\delta$ are subsets of the labels of each edge of $\delta'$. Thus all the $0$-dimensional components contained in $\delta$ are also contained in $\delta'$. $\delta'$ must contain at least one more $0$-dimensional component than $\delta$, otherwise $\delta'$ would equal $\delta$. 
This $0$-dimensional component differs from any other $0$-dimensional component in $\delta$ by at least one circuit of differing partial sums: consider a $0$-dimensional component in $\delta'$ that has a partial column sum that differs from the corresponding partial sum in any $0$-dimensional component in $\delta$. By Equation~(\ref{eq:rc}), at least one adjacent row or column partial sum of $\delta'$ must also differ from the corresponding partial sum in $\delta$. 
Thus, $\delta'$ has at least one new open or closed circuit of darkened edges, creating at least one new region. So $\mathcal{R}(\delta')\geq \omega+1$.
\end{proof}

We now define a map, which we show in Theorem~\ref{th:g_bijection} gives a bijection between faces of $P(m,n)$ and components of $\overline{\Gamma}_{(m,n)}$. 
\begin{definition}
\label{def:g(F)}
Given a collection of sign matrices $\mathcal{M}=\{M_1,M_2,\dots,M_q\}$, 
we define the map $g(\mathcal{M})=\displaystyle\bigcup_{i=1}^q g(M_i)$, where $g(M_i)$ is as in Definition~\ref{def:0-dim}.
\end{definition}

\begin{theorem}
\label{th:g_bijection}
Let $F$ be a face of $P(m,n)$ and $\mathcal{M}(F)$ be equal to the set of sign matrices that are vertices of $F$.
The map $\psi:F\mapsto g(\mathcal{M}(F))$ 
is a bijection between faces of $P(m,n)$ and components of $\overline{\Gamma}_{(m,n)}$.
\end{theorem}

\begin{proof}
Let $F$ be a face of $P(m,n)$. Then $g(\mathcal{M}(F))$ is a component of $\overline\Gamma_{(m,n)}$ since $g(\mathcal{M}(F))=\displaystyle\bigcup_{i=1}^q g(M_i)$ is a union of 0-dimensional components.
We now construct the inverse of $\psi$, call it $\varphi$. Given a component $\nu$ of $\overline\Gamma_{(m,n)}$, let $\varphi(\nu)$ be the face that results as the intersection of the facets corresponding to the not darkened edges of $\nu$. 

We wish to show $\psi(\varphi(\nu))=\nu$. First, we show $\nu\subseteq \psi(\varphi(\nu))$. Let $M$ be a sign matrix such that $g(M)$ is a $0$-dimensional component of $\nu$. $M$ is in the intersection of the facets that yields $\varphi(\nu)$, since otherwise $g(M)$ would not be a $0$-dimensional component of $\nu$. Thus $g(M)$ is in $\psi(\varphi(\nu))$ as well. So $\nu\subseteq \psi(\varphi(\nu))$, which means the edge labels of $\nu$ must be subsets of the edge labels of $\psi(\varphi(\nu))$.  

Next, we show $\nu=\psi(\varphi(\nu))$. Suppose not. Then there exists some edge $e$ of $\Gamma_{(m,n)}$ whose label in $\psi(\varphi(\nu))$  strictly contains the  label of $e$ in $\nu$. Suppose $e$ is a horizontal edge, then the  label of $e$ in $\nu$ is $0$ and the  label of $e$ in $\psi(\varphi(\nu))$ is $\star$. Then the facet corresponding to the label 0 on $e$ would have been one of the facets intersected to get $\varphi(\nu)$. Therefore the matrix partial row sum corresponding to edge $e$ would be fixed as $0$ in each sign matrix in $\varphi(\nu)$. So in the union $\psi(\varphi(\nu))$, this edge label would be the union of the edge labels of all the sign matrices in $\varphi(\nu)$, and this union would be $0$. This is a contradiction. Now suppose $e$ is a vertical edge. Then the label of $e$ in $\nu$ is $0$ or $1$ and the label of $e$ in $\psi(\varphi(\nu))$ is $\{0,1\}$. Let $\gamma$ denote the label of $e$ in $\nu$. As in the previous case, the facet corresponding to the label $\gamma$ on $e$ would have been one of the facets intersected to get $\varphi(\nu)$. Therefore the matrix partial column sum corresponding to edge $e$ would be fixed as $\gamma$ in each sign matrix in $\varphi(\nu)$. So in the union $\psi(\varphi(\nu))$, that edge label would be the union of the edge labels of all the sign matrices in $\varphi(\nu)$, and this union would be $\gamma$. This is a contradiction. Thus $\nu=\psi(\varphi(\nu))$.
\end{proof}

\begin{figure}[htbp]
\begin{center}
\includegraphics[scale=.64]{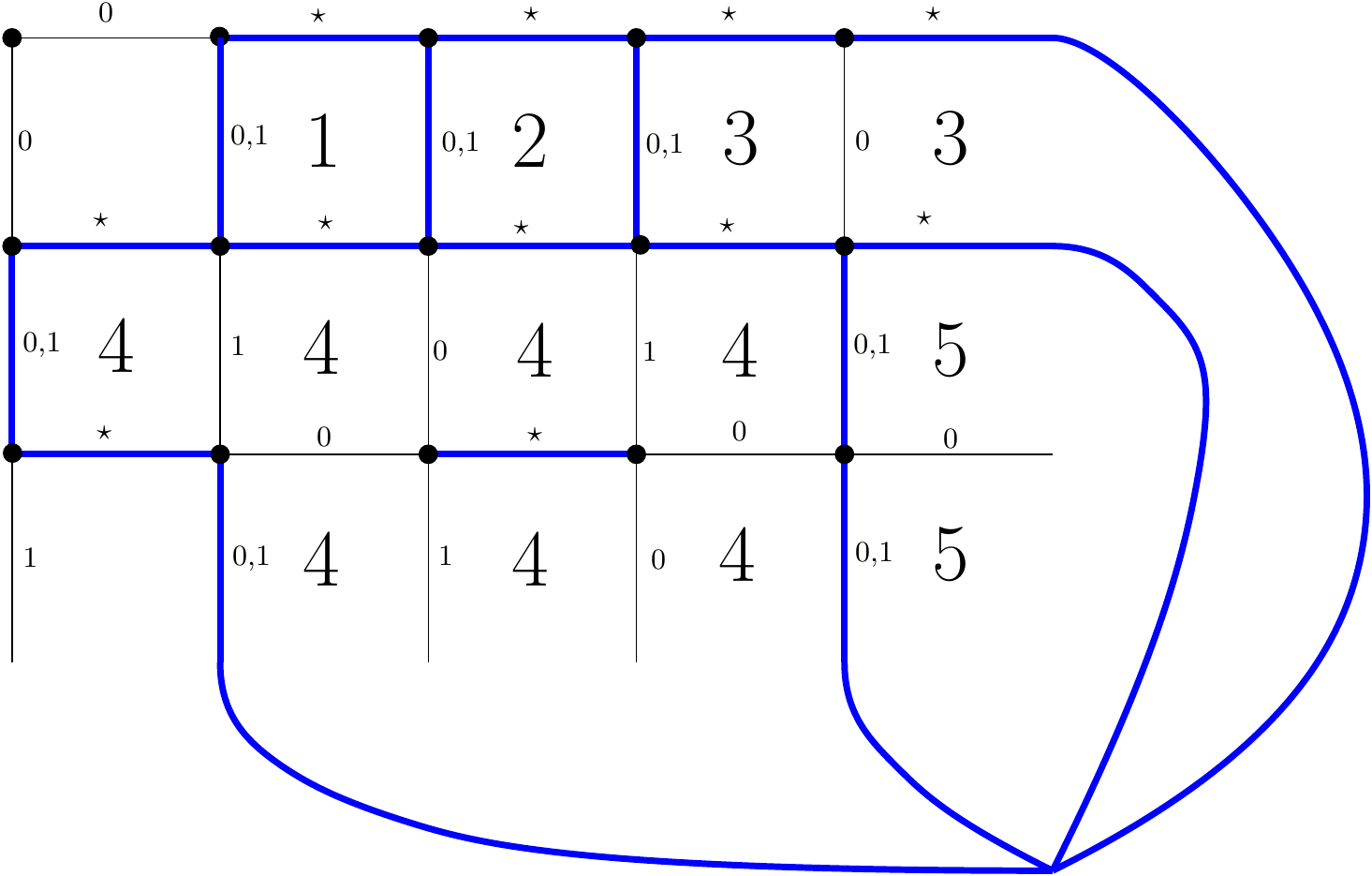}
\end{center}
\caption{A 5-dimensional face of $P(3,5)$, where the five regions are numbered.}
\label{fig:region}
\end{figure}

\begin{theorem}
\label{thm:poset_iso}
$\psi$ is a poset isomorphism. Moreover, the dimension of a face of $P(m,n)$ equals the number of regions of the corresponding component of $\overline\Gamma_{(m,n)}$. That is, for every face $F$ in $P(m,n)$,
\[ dim\; F= \mathcal{R}(\psi(F)).\]
\end{theorem}

\begin{proof}
Let $F_1$ and $F_2$ be faces of $P(m,n)$ such that $F_1 \subseteq F_2$. Then $F_1$ is an intersection of $F_2$ and some facet hyperplanes. In other words, $F_1$ is obtained from $F_2$ by setting one of the inequalities in Theorem~\ref{thm:ineqthm} to an equality. We have that $\psi(F_1)$ is obtained from $\psi(F_2)$ by changing at least one darkened edge to a non-darkened edge. Therefore we have $\psi(F_1) \subseteq \psi(F_2)$.

Conversely, suppose that $\psi(F_1) \subseteq \psi(F_2)$. Recall the inverse of $\psi$ is $\varphi$, where for any component $\nu$ of $\Gamma_{(m,n)}$, $\varphi(\nu)$ is the face of $P(m,n)$ that results as the intersection of the facets corresponding to the not darkened edges of $\nu$. Now if $\psi(F_1) \subseteq \psi(F_2)$, the darkened edges of $\psi(F_1)$ are a subset of the darkened edges of $\psi(F_2)$, so the not darkened edges of $\psi(F_2)$ are a subset of the not darkened edges of $\psi(F_1)$. So $\varphi(\psi(F_1))$ is an intersection of the facets intersected in $\varphi(\psi(F_2))$ and some additional facets (if $F_1\neq F_2$). Thus $F_1=\varphi(\psi(F_1))\subseteq \varphi(\psi(F_2))=F_2$.

Now, we prove the dimension claim. Recall that dim$(P(m,n))=mn$. Since $\psi$ is a poset isomorphism, $\psi$ maps a maximal chain of faces $F_0 \subset F_1 \subset \cdots \subset F_{mn}$  to the maximal chain $\psi(F_0) \subset \psi(F_1) \subset \cdots \subset \psi(F_{mn})$ in the components of $\overline\Gamma_{(m,n)}$. We know that the maximal component of $\Lambda_{(m,n)}$ has $mn$ regions, thus the result follows by Lemma~\ref{prop:regions} and by noting that for components $\nu$ and $\nu'$, $\nu \subsetneq \nu'$ implies $\mathcal{R}(\nu) < \mathcal{R}(\nu')$.
\end{proof}

We now discuss the face lattice of $P(\lambda,n)$. We will restate the main result in this new setting, but since most of the definitions and proofs are exactly analogous, we only note where additional notation or arguments are needed.

\begin{definition}
\label{def:completegraph_shape}
Define the \emph{shape-complete partial sum graph} denoted $\overline\Gamma_{(\lambda,n)}$ as the following labeling of the graph $\Gamma_{(
\lambda_1,n)}$. The vertical edges are labeled $\{0,1,\{0,1\}\}$ as before. The horizontal edges are labeled with the fixed row sum $\{0,\star\}$, except the last horizontal edge in row $i$ is labeled with $a_{\lambda_1-i+1}$.
An example is shown in Figure~\ref{fig:shapecomplete}.
\end{definition}

\begin{figure}[htbp]
\begin{center}
\includegraphics[scale=.8]{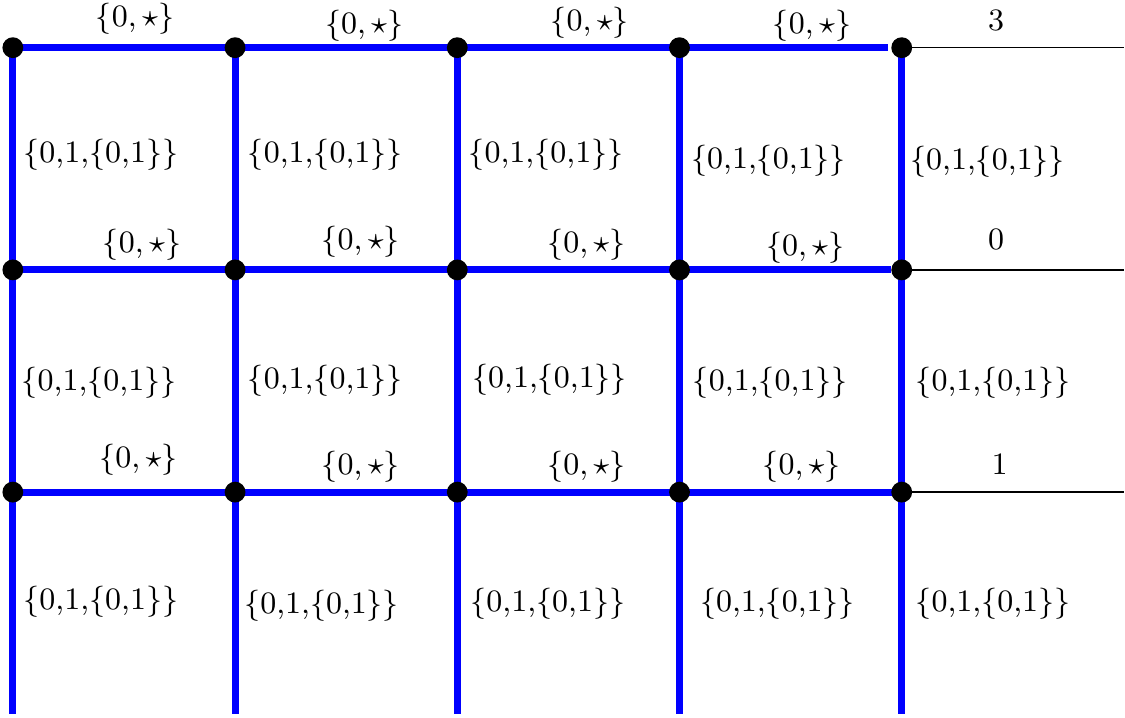}
\end{center}
\caption{The shape-complete partial sum graph of $P([3,3,3,1],5)$.}
\label{fig:shapecomplete}
\end{figure}

\begin{remark}
$0$-dimensional components, components, containment of components, and regions are defined analogously.
Let $\Lambda_{(\lambda,n)}$ denote the {partial order} on components of $\overline\Gamma_{(\lambda,n)}$  by containment.
\end{remark}

See Figure~\ref{fig:shaperegion} for an example of a component of $\Lambda_{(\lambda,n)}$.

\begin{remark}
\label{remark:max_shape}
Note the maximal component of $\Lambda_{(\lambda,n)}$ is the union of all $0$-dimensional components. Thus, it has labels $\{0,1\}$ on the vertical edges of $\Gamma_{(\lambda_1,n)}$ and $\star$ on the horizontal edges, but with the fixed row sums in the $n$th column.
\end{remark}

\begin{theorem}
\label{th:g_bijection_shape}
Let $F$ be a face of $P(\lambda,n)$ and $\mathcal{M}(F)$ be equal to the set of sign matrices that are vertices of $F$. The map $\psi:F\mapsto g(\mathcal{M}(F))$ is a bijection between faces of $P(\lambda,n)$ and components of $\overline{\Gamma}_{(\lambda,n)}$. Moreover, $\psi$ is a poset isomorphism, and the dimension of $F$ is equal to the number of regions of $\psi(F)$.
\end{theorem}

\begin{proof}
The proof is analogous to the proofs of Theorems~\ref{th:g_bijection} and \ref{thm:poset_iso}; we need only check that the number of regions of the maximal component of $\Lambda_{(\lambda,n)}$ matches the dimension of $P(\lambda,n)$.
Recall from Proposition~\ref{prop:dim} that the dimension of $P(\lambda,n)$ equals $\lambda_1(n-1)$ when $1 \leq k < n$, and $(\lambda_1 - \lambda_n)(n-1)$ when $k=n$. Note that when $1 \leq k < n$,
there are $\lambda_1(n-1)$ regions in the maximal component of $\Lambda_{(\lambda,n)}$. When $k=n$ the column partial sums in the last $\lambda_n$ rows of $\Gamma_{\lambda,n}$ are all fixed to be one, due to the first $\lambda_n$ columns of the tableau being $1,\ldots, n$. Thus there will be no darkened vertical edges in the bottom $\lambda_n$ rows, so these edges will not bound regions. So there will be $(\lambda_1 - \lambda_n)(n-1)$ regions in the maximal component of $\Lambda_{(\lambda,n)}$.
\end{proof}

\begin{figure}
\begin{center}
\includegraphics[scale=.6]{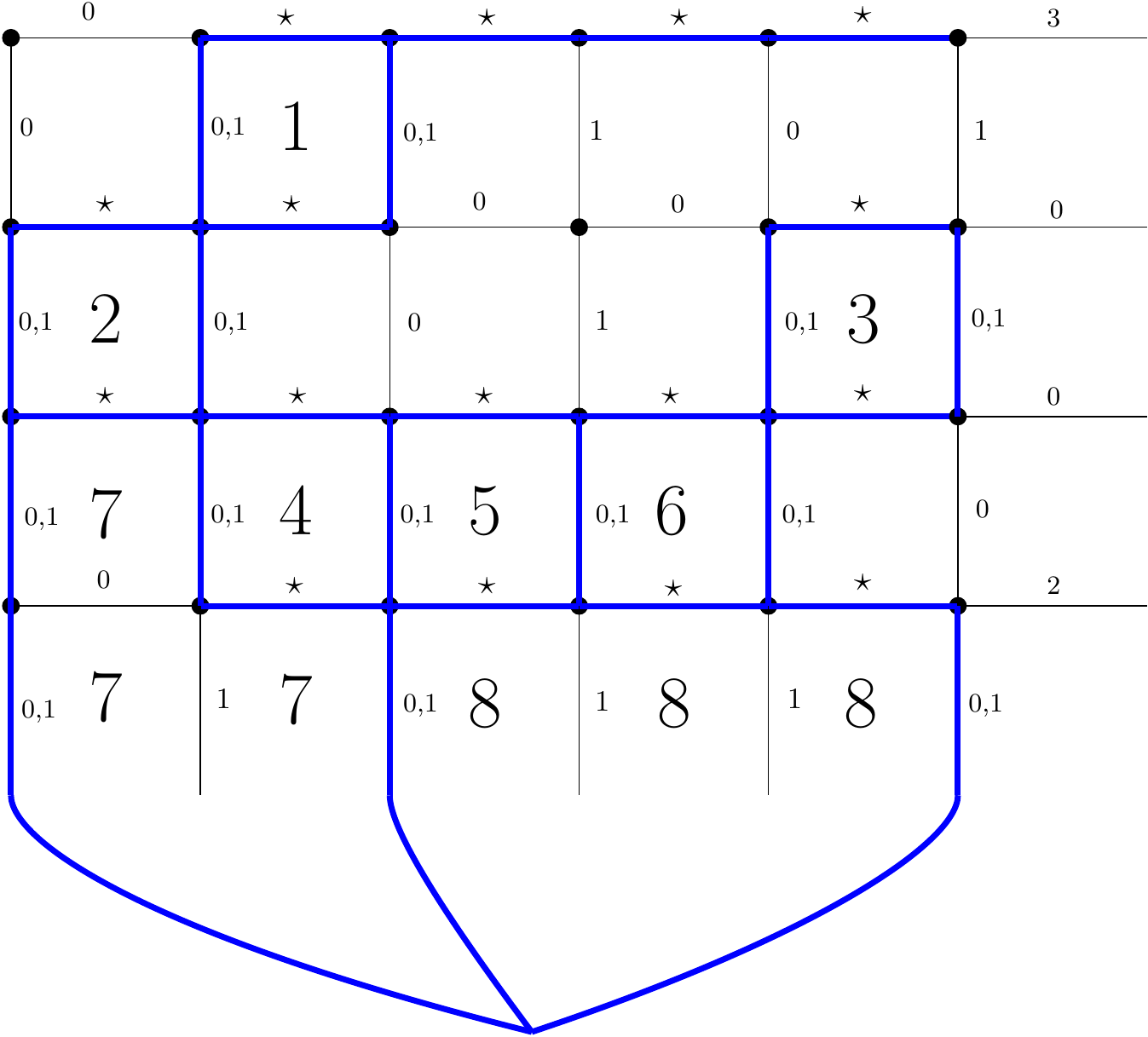}
\end{center}
\caption{An $8$-dimensional component of $P([4,4,4,1,1],6)$.}
\label{fig:shaperegion}
\end{figure}

\section{Connections and related polytopes}
\label{sec:connections}
In this section, we describe connections between sign matrix polytopes and related polytopes. First we describe how $P(\lambda,n)$ and $P(m,n)$ are related. We will need a few additional definitions in order to relate $P(\lambda,n)$ to $P(m,n)$ when $\lambda_1 < m$.

\begin{definition}
\label{def:lambda1lessthanm}
Fix $\lambda$ and $m$ such that $\lambda_1 \leq m$.
Let $M_m(\lambda,n)$ be the set of $m\times n$ sign matrices $M=(M_{ij})$ such that: 
\begin{align}
\label{eq:Mij_rowsum2}
M_{ij} &=0 & \mbox{ for all } 1 \leq i\leq m-\lambda_1 \\
\label{eq:Mij_rowsum3}
\displaystyle\sum_{j=1}^n M_{ij} &= a_{\lambda_1-(i-(m-\lambda_1))+1},
 & \mbox{ for all }m-\lambda_1+1\le i \le m.
\end{align}
Let \emph{$P_m(\lambda,n)$} be the polytope defined as the convex hull, as vectors in $\mathbb{R}^{m n}$, of all the matrices in $M_m(\lambda,n)$.
\end{definition}

Note that if $\lambda_1=m$, $M_m(\lambda,n)=M(\lambda,n)$ so that $P_m(\lambda,n)=P(\lambda,n)$.

\begin{remark}
The only difference between $M_m(\lambda,n)$ and $M(\lambda,n)$ is that we have inserted $m-\lambda_1$ additional rows of zeros at the top of each matrix. Therefore, $P_m(\lambda,n)$ and $P(\lambda,n)$ have all the same combinatorial properties (dimension, face lattice, volume, etc.); the only difference is their ambient dimensions.  In particular, a slight modification of the bijection of Theorem~\ref{thm:MtoSSYT} shows $M_m(\lambda,n)$ is also in bijection with $SSYT(\lambda,n)$.
\end{remark}

We give below an inequality description of $P_m(\lambda,n)$, whose proof follows immediately from Theorem~\ref{thm:ineqthmshape} and Definition~\ref{def:lambda1lessthanm}.
\begin{corollary}
\label{thm:ineqthmshape_m}
$P_m(\lambda,n)$ consists of all 
$m\times n$ real matrices $X=(X_{ij})$ such that:
\begin{align}
\label{eq:eq1m}
0 \leq \displaystyle\sum_{i'=1}^{i} X_{i'j} &\leq 1, &\mbox{ for all }1 \leq i\leq m, 1\le j\le n \\
\label{eq:eq2m}
0 \leq \displaystyle\sum_{j'=1}^{j} X_{ij'}, &
&\mbox{ for all }1 \leq i\leq m, 1\le j\le n \\
\label{eq:eq3m}
\displaystyle\sum_{j'=1}^n X_{ij'} &= a_{\lambda_1-(i-(m-\lambda_1))+1},
 & \mbox{ for all } m-\lambda_1+1\le i \le m \\
 \label{eq:eq4m}
X_{ij} &=0 & \mbox{ for all } 1 \leq i\leq m-\lambda_1, 1\le j\le n. 
\end{align}
\end{corollary}

\smallskip

\begin{lemma}
\label{thm:hyper}
$P_m(\lambda,n)$ is the intersection of a ${\lambda_1 (n-1)}$--dimensional affine subspace of $\mathbb{R}^{mn}$ and $P(m,n)$.
\end{lemma}

\begin{proof}
The only differences between the inequality descriptions 
of $P_m(\lambda,n)$ and $P(m,n)$ are (\ref{eq:eq3m}) and (\ref{eq:eq4m}). (\ref{eq:eq4m}) fixes the first $m-\lambda_1$ matrix rows to contain all zeros, while
(\ref{eq:eq3m}) fixes the remaining row total sums in $P_m(\lambda,n)$.
So $P_m(\lambda,n)$ is the intersection of $P(m,n)$ and the affine subspace defined by (\ref{eq:eq3m}) and (\ref{eq:eq4m}).
\end{proof}

See Figure~\ref{fig:cubes} for an example.
\smallskip

\begin{figure}[htbp]
\begin{center}
\includegraphics[scale=1.2]{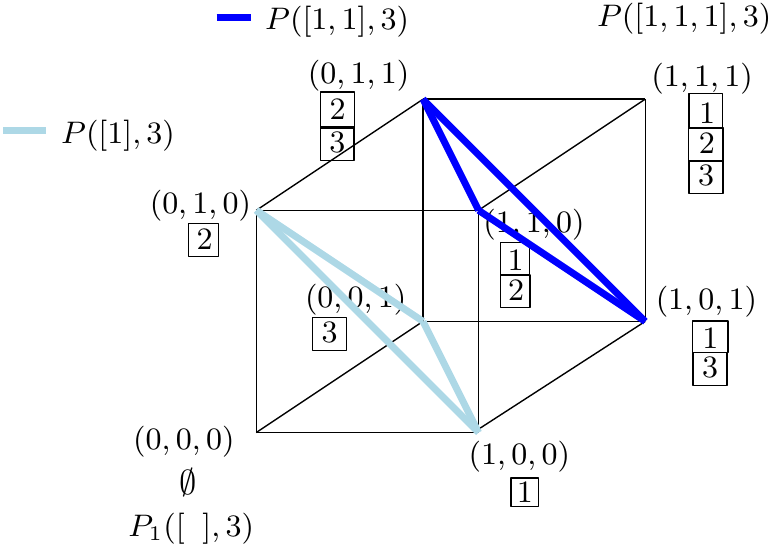}
\end{center}
\caption{The cube above is $P(1,3)$; the $P(\lambda,3)$ polytopes for each partition of shape $\lambda$ in a $1\times 3$ box are also indicated. 
$P_1([~],3)$ and $P([1,1,1],3)$ are each a single point, while $P([1],3)$ and $P([1,1],3)$ are the indicated triangles cutting through $P(1,3)$.
}
\label{fig:cubes}
\end{figure}

Alternating sign matrices are a motivating special case of sign matrices.  We give the usual definition below and then relate it to sign matrices.
\begin{definition}[\cite{MRRASM}]
\label{def:asm}
An \emph{alternating sign matrix} is a square matrix with entries in $\left\{-1,0,1\right\}$ such that the rows and columns each sum to one and the nonzero entries along any row or column alternate in sign. Let $A(n)$ denote the set of $n\times n$ alternating sign matrices.
\end{definition}

The following lemma is implicit in Aval's paper on sign matrices.
\begin{lemma}[\cite{aval}]
\label{prop:asm_sign_matrix}
$A(n)$ is the set of sign matrices $M=\left(M_{ij}\right)$   
in $M([n,n-1,\ldots,2,1],n)$ 
satisfying the additional requirement:
\begin{align}
\label{eq:asm2}
 &\displaystyle\sum_{j'=1}^{j} M_{ij'} \in\{0,1\} \mathrm{~for~all~} i,j.
\end{align}
\end{lemma}

\begin{proof}
Let $M\in A(n)$. Then the nonzero entries of $M$ alternate between $1$ and $-1$ across any row or column. The first nonzero entry in a row or column must be a $1$, since otherwise that row or column would not sum to $1$. Thus (\ref{eq:sm1}) and (\ref{eq:sm2}) from Definition~\ref{def:sm} of a sign matrix and (\ref{eq:asm2}) above are satisfied. Also in an alternating sign matrix all of the total row sums are $1$. Recall from (\ref{eq:Mij_rowsum}) that the row sums of a sign matrix equal $a_{\lambda_1-i+1}$, so since each row sum of $M$ is $1$, $M$ must be in $M([n,n-1,\ldots,2,1],n)$.

Now let $M\in M([n,n-1,\ldots,2,1],n)$ satisfy (\ref{eq:asm2}). $M$ is an $n\times n$ matrix whose rows each sum to $1$ since $M\in M([n,n-1,\ldots,2,1],n)$. By (\ref{eq:sm1}) and the fact that the sum of all the matrix entries is $n$, we have that the columns must each sum to $1$. Then (\ref{eq:sm1}) and (\ref{eq:asm2}) imply that the nonzero entries of $M$ alternate in sign along each row and column.
\end{proof}

\begin{remark}
\label{re:ASMstairs}
It is well-known (see e.g.~\cite{MRRASM}) that alternating sign matrices are in bijection with \emph{monotone triangles}, which are equivalent (by rotation) to
semistandard Young tableau of staircase shape with first column $(1,2,\ldots,n)$ and such that each northeast to southwest diagonal is weakly increasing. This bijection is a specialization of the bijection of Theorem~\ref{thm:MtoSSYT}.
\end{remark}

\begin{definition}[\cite{behrend,striker}]
The $n$th alternating sign matrix polytope, denoted $ASM_n$, is the convex hull of all the $n \times n$ alternating sign matrices, considered as vectors in $\mathbb{R}^{n^2}$. 
\end{definition} 

\begin{remark} Striker~\cite{striker} and Behrend and Knight~\cite{behrend} independently defined and proved several results about $ASM_n$. The dimension of $ASM_n$ is $(n-1)^2$, an inequality description of $ASM_n$ is that the rows and columns sum to $1$ and the partial sums are between $0$ and $1$, and the vertices of $ASM_n$ are all the $n\times n$ alternating sign matrices~\cite{behrend,striker}. $ASM_n$ has $4[(n-2)^2+1]$ facets and a nice face lattice description~\cite{striker}. These are a few of the results that inspired the research of this paper.

Some further properties of $ASM_n$ were studied by Brualdi and Dahl~\cite{BrualdiDahl}.
These include results regarding edges of $ASM_n$, an alternative
proof of the characterization of the vertices of $ASM_n$, and an alternative
proof of the linear characterization of $ASM_n$.
\end{remark}

We see the connection between $P(\lambda,n)$ and $ASM_n$ in the following theorem.

\begin{lemma}
\label{thm:staircase}
$P([n, n-1, \cdots, 2, 1],n)$ contains $ASM_n$.
\end{lemma}

\begin{proof}
Lemma~\ref{prop:asm_sign_matrix} gives that the set of $n\times n$ alternating sign matrices is a subset of $M([n, n-1, \cdots, 2, 1],n)$. So the convex hull of $n\times n$ alternating sign matrices will be contained in the convex hull of $M([n, n-1, \cdots, 2, 1],n)$, which is $P([n, n-1, \cdots, 2, 1],n)$.
\end{proof}

The Birkhoff polytope contains no lattice points except the permutation matrices, which are its vertices. 
We show something similar happens in the case of sign matrices and alternating sign matrices.

\begin{theorem}
\label{prop:lattice_points}
There are no lattice points in $P(m,n)$, $P(\lambda,n)$, or $ASM_n$ 
other than the matrices used to construct them. 
\end{theorem}
\begin{proof}
Let $M$ be an integer-valued matrix inside the polytope $P(m,n)$. Then $M$ fits the inequality description of $P(m,n)$. From the inequalities, all partial column sums are either $0$ or $1$, thus the entries of $M$ must be in $\{-1,0,1\}$. Also, all partial row sums are nonnegative, so $M$ satisfies the definition of an $m\times n$ sign matrix.

By Lemma~\ref{thm:hyper}, $P(\lambda,n)$ is contained in $P(\lambda_1,n)$.
By Lemma~\ref{thm:staircase}, $ASM_n$ is contained in $P([n, n-1, \cdots, 2, 1],n)$ which by Theorem~\ref{thm:hyper} is contained in $P(n,n)$. Thus, the results follow.
\end{proof}

\section{$P(v,\lambda,n)$ and transportation polytopes}
\label{sec:transportation}
Thus far in this paper, we have defined and studied the sign matrix polytope $P(m,n)$ and the polytope $P(\lambda,n)$ whose vertices are the sign matrices with row sums determined by $\lambda$. We may furthermore restrict to sign matrices with prescribed column sums; we define this polytope below, calling it $P(v,\lambda,n)$. We show in Theorem~\ref{thm:transportation} that the nonnegative part of this polytope is a transportation polytope.

\begin{definition}
Let $\lambda$ be a partition with $k$ parts and $v$ a vector of length $k$ with strictly increasing entries at most $n$.
Let $SSYT(v,\lambda,n)$ denote the set of semistandard Young tableaux of shape $\lambda$ with entries at most $n$ and first column $v$. 
\end{definition}
For example, the tableau of Figure~\ref{fig:ydssyt} is in $SSYT((1,2,3,6),[6,3,3,1],n)$ for any $n \geq 7$.

\begin{remark}
We do not know an enumeration for $SSYT(v,\lambda,n)$, though the numbers we have calculated look fairly nice.
\end{remark}

\begin{definition}
\label{def:MtoSSYTv}
Fix $\lambda$ and $n\in\mathbb{N}$ and $v$ a vector of length $k$ with strictly increasing entries at most $n$.
Let \emph{$M(v,\lambda,n)$} be the set of $M\in M(\lambda,n)$ such that:
\begin{align}
\label{eq:Mij_colsum_v}
\displaystyle\sum_{i=1}^{\lambda_1} M_{ij} &= 1, & \mbox{ if } j\in v\mbox{ and } 0\mbox{ otherwise.} 
\end{align}
\end{definition}

\begin{theorem}
$M(v,\lambda,n)$ is in explicit bijection with $SSYT(v,\lambda,n)$.
\label{prop:MtoSSYTv}
\end{theorem}

\begin{proof}
We know that $M(\lambda,n)$ is in bijection with $SSYT(\lambda,n)$ from Theorem~\ref{thm:MtoSSYT}. So we only need to check (\ref{eq:Mij_colsum_v}). Consider $M \in M(v,\lambda,n)$ and follow the bijection of Theorem~\ref{thm:MtoSSYT} to construct the corresponding $T \in SSYT(\lambda,n)$. Recall that in $M(v,\lambda,n)$, $v$ records which columns of $M$ have a total sum of $1$. Thus, the numbers in $v$ are the entries of $T$ in the first column of $\lambda$, so $T \in SSYT(v,\lambda,n)$. 

Now consider $T \in SSYT(v,\lambda,n)$ and its corresponding sign matrix $M\in M(\lambda,n)$. The first column of $T$ is fixed to be the numbers in $v$. The first column of $T$ gets mapped to the last row of $M$.
That is, for each number in the first column of $T$, the corresponding column of $M$ will sum to $1$. The rest of the columns of $M$ will sum to $0$.
Thus $M \in M(v,\lambda,n)$.
\end{proof}

\begin{definition}
Let \emph{$P(v,\lambda,n)$} be the polytope defined as the convex hull, as vectors in $\mathbb{R}^{\lambda_1 n}$, of all the matrices in $M(v,\lambda,n)$. We say this is the sign matrix polytope with row sums determined by $\lambda$ and column sums determined by $v$.
\end{definition}

We now discuss analogous properties to those proved in the rest of the paper regarding $P(m,n)$ and $P(\lambda,n)$. Since many of these proofs are very similar to proofs we have already discussed,  we only note how the proofs differ from those in the other cases.

\begin{proposition}
\label{prop:transdim}
The dimension of $P(v,\lambda, n)$ is $(\lambda_1-1)(n-1)$ if $1 \leq k < n$. When $k=n$, the dimension is $(\lambda_1 - \lambda_n)(n-1).$ 
\end{proposition}

\begin{proof}
Since each matrix in $M(v,\lambda,n)$ is $\lambda_1 \times n$, the ambient dimension is $\lambda_1 n$. 
However, when constructing the sign matrix corresponding to a tableau of shape $\lambda$, as in Theorem~\ref{thm:MtoSSYT}, the last column is determined by the shape $\lambda$ via the prescribed row sums (\ref{eq:Mij_rowsum}) of Definition~\ref{def:MtoSSYT}. The last row of the matrix is determined by $v$ using (\ref{eq:Mij_colsum_v}). These are the only restrictions on the dimension when $1 \leq k < n$, reducing the free entries in the matrix by one column and one row. Thus, 
the dimension is $(\lambda_1-1)(n-1)$. When $k=n$, it must be that $v=(1,2,\ldots,n)$ and $P(v,\lambda,n)$ equals $P(\lambda,n)$, so we reduce to this case.
\end{proof}

\begin{theorem}
\label{thm:v_lambdavertex}
The vertices of $P(v,\lambda,n)$ are the sign matrices $M(v,\lambda,n)$.
\end{theorem}
\begin{proof}
The hyperplane constructed in the proof of Theorem~\ref{thm:lambdavertex} separates a given sign matrix from all other sign matrices in $M(\lambda,n)$, which includes $M(v,\lambda,n)$.
\end{proof}

\begin{theorem}
\label{thm:v_ineqthmshape}
$P(v,\lambda,n)$ consists of all 
$\lambda_1\times n$ real matrices $X=(X_{ij})$ such that:
\begin{align}
\label{eq:eq1t}
0 \leq \displaystyle\sum_{i'=1}^{i} X_{i'j} &\leq 1, &\mbox{ for all }1 \leq i\leq \lambda_1, 1\le j\le n \\
\label{eq:eq2t}
0 \leq \displaystyle\sum_{j'=1}^{j} X_{ij'}, &
&\mbox{ for all }1\le j\le n, 1\leq i\leq \lambda_1 \\
\label{eq:eq3t}
\displaystyle\sum_{j'=1}^n X_{ij'} &= a_{\lambda_1-i+1}, &\mbox{ for all } 1\leq i \leq \lambda_1 \\
\label{eq:eq4t}
\displaystyle\sum_{i'=1}^{\lambda_1} X_{i'j} &= 1, &\mbox{ if } j\in v\mbox{ and } 0\mbox{ otherwise.} 
\end{align}
\end{theorem}
\begin{proof}
This proof follows the proof of Theorem~\ref{thm:ineqthmshape}, except since both the row and column sums are fixed, only closed circuits are needed.
\end{proof}

\begin{definition}
\label{def:v_completegraph_shape}
Define $\overline\Gamma_{(v,\lambda,n)}$ as the following labeling of the graph $\Gamma_{(
\lambda_1,n)}$. 
All edges are labeled as in $\overline\Gamma_{(\lambda,n)}$, except the last vertical edge in column $j$ is labeled $1$ if $j\in v$ and $0$ otherwise.
$0$-dimensional components, components, containment of components, and regions are defined analogously.
Let $\Lambda_{(v,\lambda,n)}$ denote the {partial order} on components of $\overline\Gamma_{(v,\lambda,n)}$  by containment.
\end{definition}

\begin{theorem}
\label{th:v_g_bijection_shape}
Let $F$ be a face of $P(v,\lambda,n)$ and $\mathcal{M}(F)$ be equal to the set of sign matrices that are vertices of $F$. The map $\psi:F\mapsto g(\mathcal{M}(F))$ is a bijection between faces of $P(v,\lambda,n)$ and components of $\overline{\Gamma}_{(v,\lambda,n)}$. 
Moreover, $\psi$ is a poset isomorphism, and the dimension of $F$ is equal to the number of regions of $\psi(F)$.
\end{theorem}
\begin{proof}
The proof is analogous to the proof of Theorem~\ref{th:g_bijection_shape}; we need only check that the number of regions of the maximal component of $\Lambda_{(v,\lambda,n)}$ matches the dimension of $P(v,\lambda,n)$.
Recall from Proposition~\ref{prop:transdim} that the dimension of $P(v,\lambda,n)$ equals $(\lambda_1-1)(n-1)$ when $1 \leq k < n$, and $(\lambda_1 - \lambda_n)(n-1)$ when $k=n$. Note that when $1 \leq k < n$,
there are $(\lambda_1-1)(n-1)$ regions in the maximal component of $\Lambda_{(v,\lambda,n)}$. When $k=n$, the only possible first column of $T\in SSYT(\lambda,n)$ is $v=(1,2,\ldots,n)$, thus $P(v,\lambda,n)=P(\lambda,n)$ and we may use Theorem~\ref{th:g_bijection_shape}.
\end{proof}

Theorem~\ref{thm:transportation} relates sign matrix polytopes to transportation polytopes. We first give the following definition (see, for example, \cite{deloera} and references therein).
\begin{definition} 
 Fix two integers $p,q \in \mathbb{Z}_{>0}$ and two vectors $y \in\mathbb{R}_{\geq 0}^p$ and $z \in \mathbb{R}_{\geq 0}^q$. The \emph{transportation polytope} $P_{(y,z)}$ 
is the convex polytope defined in the $pq$ variables
$X_{ij} \in \mathbb{R}_{\geq 0}$, $1\leq i\leq p, 1\leq j \leq q$) satisfying the $p + q$ equations:
\begin{align}
\label{eq:transp1}
\sum_{j'=1}^{q} X_{ij'} &= y_i, &  \mbox{ for all } 1\leq i\leq p  \\
\label{eq:transp2}
\sum_{i'=1}^{p} X_{i'j} &=z_j, &  \mbox{ for all } 1\leq j \leq q.  \end{align}
\end{definition}

\begin{theorem}
\label{thm:transportation}
The nonnegative part of $P(v,\lambda,n)$ is the transportation polytope $P_{(y,z)}$, where $y_i=a_{\lambda_1-i+1}$ for all $1\leq i\leq \lambda_1$ and $z_j=$ 1 if $j\in v$ and $0$ otherwise.
\end{theorem}
\begin{proof}
By Theorem~\ref{thm:v_ineqthmshape}, the nonnegative part of $P(v,\lambda,n)$ is contained in $P_{(y,z)}$, since for these choices of $y$ and $z$, (\ref{eq:transp1}) and (\ref{eq:transp2}) are exactly (\ref{eq:eq3t}) and (\ref{eq:eq4t}). For the reverse inclusion, note in addition that any matrix with nonnegative entries and column sums at most $1$ satisfies (\ref{eq:eq1t}) and (\ref{eq:eq2t}). \end{proof}

This is analogous to the fact that the non-negative part of the alternating sign matrix polytope is the {Birkhoff polytope}~\cite{behrend,striker}.

\section*{Acknowledgments}
The authors thank Jesus De Loera for helpful conversations on transportation polytopes, Dennis Stanton for making us aware of Theorem~\ref{thm:Gordon}, and the anonymous referee for helpful comments.  The authors also thank the developers of \verb|SageMath|~\cite{sage} software, especially the code related to polytopes and tableaux, which was helpful in our research, and the developers of \verb|CoCalc|~\cite{SMC} for making \verb|SageMath| more accessible. JS was supported by a grant from the Simons Foundation/SFARI (527204, JS).

\bibliographystyle{plain}
\bibliography{biblio}

\end{document}